\documentclass[11pt]{amsart}
\usepackage{amssymb,amsthm,amsmath,geometry,tikz,setspace,enumerate,enumitem,bbm}
\usepackage[textsize=footnotesize,color=blue]{todonotes}
\usepackage[all]{xy}
\newtheorem{theorem}{Theorem}
\newtheorem{proposition}[theorem]{Proposition}
\newtheorem{lemma}[theorem]{Lemma}
\newtheorem{corollary}[theorem]{Corollary}

\newtheorem{claim}{Claim}

\newtheorem{example}{Example}

\newcommand{\sm}{\setminus}
\newcommand{\N}{\mathbb{N}}
\newcommand{\cFkk}{\mathcal{F}_{k+1}}
\newcommand{\cFkkk}{\mathcal{F}_{k+2}}
\newcommand{\cFk}{\mathcal{F}_{k}}

\newcommand{\eps}{\varepsilon}
\newcommand{\Pro}{\mathbb{P}}
\newcommand{\Exp}{\mathbb{E}}

\newcommand{\cA}{\mathcal{A}}

\newcommand{\cD}{\mathcal{D}}

\newcommand{\cG}{\mathcal{G}}
\newcommand{\cF}{\mathcal{F}}

\newcommand{\cE}{\mathcal{E}}

\newcommand{\cH}{\mathcal{H}}
\newcommand{\Td}{\Sigma^{\mathcal{D}}}
\newcommand{\ind}[1]{\mathbf{1}_{\{#1\}}}

\newcommand{\fD}{\mathfrak{D}}

\newcommand{\fS}{\mathfrak{S}}

\newcommand{\bP}{\mathbb{P}}
\newcommand{\bE}{\mathbb{E}}
\newcommand{\bN}{\mathbb{N}}
\newcommand{\Bin}{\text{Bin}}
\newcommand{\avd}{\bar{d}}
\newcommand{\cri}{{crit}}

\newcommand{\new}[1]{\textcolor{red}{#1}}

\title{Percolation on random graphs with a fixed degree sequence}
\author{Nikolaos Fountoulakis, Felix Joos
~and Guillem Perarnau}
\date{}
\thanks{Felix Joos was supported by the EPSRC, grant no. EP/M009408/1.}
\thanks {Guillem Perarnau was supported by Ministerio de Econom\'ia y Competitividad grant MTM2017-82166-P and the  Ministerio de Ciencia e Innovaci\'on grant PID2020-113082GB-I00.}

\begin{document}
\maketitle

\begin{abstract}
We consider bond percolation on random graphs with given degrees and bounded average degree. In particular, we consider the order of the largest component 
after  the random deletion of the edges of such a random graph. We give a rough characterisation of those degree 
distributions for which bond percolation with high probability leaves a component of linear order, known usually as a \emph{giant component}. We show that 
essentially the critical condition has to do with the tail of the degree distribution.  
Our proof makes use of recent technique 
which is based on \emph{the switching method} and avoids the use of the classic configuration model on degree sequences that have a limiting distribution.  
Thus our results hold for sparse degree sequences without the usual restrictions that accompany the configuration model.
\medskip

\noindent
\texttt{keywords}: random graphs with given degrees, bond percolation, giant component, power law

\noindent
\emph{2010 AMS Subj. Class.}: 05C80, 05C82
\end{abstract}


\section{Introduction}

Random graphs with a given degree sequence have become an integral part of the theory of random graphs. 
Let $n\geq 2$ and let $\cD=(d_1,\dots,d_n)$ be a degree sequence of length $n$; 
that is, a vector of non-negative integers which represent the degrees of the set of vertices $[n]:=\{1,\ldots, n\}$. 
In other words, vertex $i$ has degree $d_i$ for each $i\in [n]$. 
Without loss of generality, we assume that $d_1\leq \ldots \leq d_n$. In fact, our results 
deal with properties that are closed under automorphisms and remain valid when relabelling the vertex set.
If not stated otherwise, we will assume that $d_1\geq 1$. The results for degree sequences containing vertices of degree $0$ can be easily deduced from the analysis of degree sequences without them.
We will also assume that $\cD$ is feasible; that is, there exists at least one graph with degree sequence $\cD$.  
The main object of our study is $G^\cD$, which is a graph chosen uniformly at random among all \emph{simple} graphs on $[n]$ having degree sequence $\cD$.

Random graphs with a given degree distribution appear also in the context of graph enumeration. 
Bender and Canfield~\cite{BenCan}, 
as well as Bollob\'as~\cite{BolEnum} and Wormald~\cite{Worm78}, came up with the now well-known \emph{configuration model}, 
which has become a standard tool in the analysis of random graphs that are sampled uniformly from the set of all simple
graphs with a given degree sequence. 
However, the study of such random graphs through the configuration model has some limitations, as it often requires bounds on the growth of the maximum degree of the degree sequence. Typically, these are implicitly imposed by bounds on the second (or higher) moment of the degree sequence.

In 1995, Molloy and Reed~\cite{molloy1995critical} investigated the component structure of 
$G^\cD$ and, more specifically, the emergence of the \emph{giant component} (a component containing at least a constant fraction of the vertices).
This is one of the central questions in the theory of random graphs.
They provided a condition on $\cD$ that characterises the emergence of a giant component in $G^\cD$ given that $\cD$ satisfies a number of 
technical conditions. 
This result 
has been widely applied to the analysis of a variety of complex networks~\cite{aiello2000random, albert2002statistical, boccaletti2006complex, newman2003structure} and there are several refinements of it~\cite{bollobas2015old, hatami2012scaling, janson2009new,  kang2008critical, molloy1998size}. 
The technical restrictions on $\cD$ in~\cite{molloy1995critical} result from the use of the configuration model.
These restrictions have been weakened in subsequent papers~\cite{bollobas2015old, janson2009new}. 
Recently, Joos, Perarnau, Rautenbach, and Reed~\cite{joos2016how} managed to completely remove all restrictions on $\cD$ by using an analysis based on the switching method. This provides a new criterion for the existence of a giant component in $G^\cD$ that can be applied to every degree sequence.

In this paper, we follow this novel approach and consider the~\emph{component sizes} of a random graph with a given degree sequence under the random deletion of its edges. 
For a graph $G$ and a real number  $p\in [0,1]$, 
we denote by $G_p$ the random subgraph of $G$ in which every edge of $G$ is retained independently with probability $p$.  This is commonly known as \emph{bond percolation} on $G$. 
The theme of this paper is the component structure of $G^\cD_p$.  
Since  $G^\cD$ itself is a random graph, 
$G^\cD_p$ should be understood as follows: 
first, we choose a graph $G^\cD$ uniformly at random from the set of all simple graphs with degree sequence $\cD$ 
and thereafter each edge of $G^\cD$ is retained independently with probability $p$. 

The structure of $G^\cD$ 
has been studied in great detail for the case $d_i=d$ for all $i\in [n]$ and some $d\in \N$
(in this case we also write $G(n,d)$ for $G^\cD$).
The bond percolation of $G(n,d)$ was first studied by Goerdt~\cite{goerdt2001giant}. 
He proved that there exists a critical value $p_\cri=1/(d-1)$ such that 
the existence of a giant component depends on whether $p<p_\cri$ or $p>p_\cri$. 
Bond percolation of $G(n,d)$ near the critical probability $p_\cri$ has been extensively studied~\cite{nachmias2010critical, pittel2008edge, riordan2012phase}.
The first author~\cite{fountoulakis2007percolation} and Janson~\cite{janson2008percolation} considered the bond percolation of $G^\cD$, proving the existence of a critical probability, provided that $\cD$ satisfies some technical conditions, similar to those required in~\cite{molloy1995critical}. 
These results have been extended to a more general setting by Bollob\'as and Riordan~\cite{bollobas2015old}. 


In the present work, we determine those conditions on $\cD$ which ensure that $p_\cri$ is bounded away from 0.    
We consider arbitrary degree sequences without restrictions as in~\cite{bollobas2015old,fountoulakis2007percolation,janson2008percolation,molloy1995critical} and we only insist that the total number of edges grows linearly with the number of vertices $n$. 
We call those sequences \emph{sparse}. 
(We briefly discuss the non-sparse case at the end of the paper.)
Besides the mathematical motivation, 
sparse graphs are also the main focus in the theory of complex networks as this is a property that is observed in 
several networks that arise in applications~\cite{albert2002statistical}.



Consider a sequence of degree sequences $\fD= (\cD_n)_{n\geq 2}$, where $\cD_n=(d_1^{(n)}, \dots,d_n^{(n)})$. 
Let $D_n$ be the random variable that is the degree in $\cD_n$ of a vertex selected uniformly at random. 
For $c\in\N$, we define $W(c,\cD) := \{ i \ : \  d_i \geq c \}$; that is, $W(c,\cD)$ is the set of vertices of degree at least $c$ and  set $W_n(c):=W(c,\cD_n)$.
The sequence $(D_n)_{n\geq 2}$ is \emph{uniformly integrable} if for every $\eps>0$, there exist $c,n_0$ such that for every $n\geq n_0$, we have
\begin{align}\label{eq:UI_seq}
\sum_{i\in W_n(c)} d_i^{(n)} \leq \eps n,
\end{align}
and \emph{strongly uniformly integrable} if for every $\eps>0$, there exists $c_0$ such that for any $c\geq c_0$ there exists $n_0$ such that for every $n\geq n_0$, we have
\begin{align}\label{eq:SUI_seq}
\sum_{i\in W_n(c)} d_i^{(n)} \leq \frac{\eps}{c}\cdot n.
\end{align}
Strong uniform integrability can be seen as a weaker version of  bounded second moment conditions.
We say that $(D_n)_{n\geq 2}$ \emph{robustly fails the strong uniform integrability assumption} if there exists a function $f : \mathbb{R} \to \mathbb{R}$, such that $f (x)\to \infty$ as $x\to \infty$, and a $c_0>0$ with the property
that for any $c\geq c_0$ there exists $n_0$ such that for all $n\geq n_0$ we have 
\begin{align}\label{eq:RFSUI_seq}
\sum_{i\in W_n(c)} d_i^{(n)} > \frac{f(c)}{c}\cdot n.
\end{align}

For a graph $G$, we denote by $L_1(G)$ the number of vertices in the largest component (ties are  resolved following the lexicographic ordering of the vertices).

We now state our main result in the context of sequences of degree sequences that satisfy a \emph{mild} convergence condition. 
\begin{theorem}\label{thm:seq}
Suppose that $\avd \geq 1$ and let $\fD=(\cD_n)_{n\geq 2}$ be a sequence of degree sequences such that for all $n \geq 2$ the average degree of $\cD_n$ is at most $\avd$.
Suppose that 
$$ 
d=d(\fD):=\sup_{c\geq 1} \lim_{n \to \infty} \max\left\{\frac{\sum_{i \in V \sm W_n(c)} d_i^{(n)} (d_i^{(n)}-1)}{\sum_{i \in V \sm W_n(c)} d_i^{(n)}} ,1\right\}
$$
exists.
Let $D_n$ be the degree in $\cD_n$ of a vertex chosen uniformly at random.
Then
\begin{enumerate}[label=(\roman*)]
\item \label{item:1} if the sequence $(D_n)_{n \geq 2}$ is strongly uniformly integrable and $d<\infty$, then for every $\epsilon>0$ the following hold:
\begin{itemize}
\item[-] if $0\leq p\leq (1-\epsilon)\frac{1}{d}$, then 
$$
\Pro[L_1(G_p^{\cD_n})=o(n)]=1-o(1)\;;
$$
\item[-] if $(1+\epsilon)\frac{1}{d}\leq p\leq 1$, then there exists $\rho=\rho(\epsilon)$ such that 
$$
\Pro[L_1(G_p^{\cD_n})>\rho n]=1-o(1)\;.
$$
\end{itemize}
\item\label{item:2} if either $d=\infty$ or the sequence robustly fails the strong uniform integrability assumption, then for all $p,\delta \in (0,1)$, there exist $\rho>0$ and $n_0$ such that for all $n\geq n_0$, we have
 $$
 \Pro[L_1(G_p^{\cD_n})>\rho n]\geq 1-\delta\;.
 $$
\end{enumerate}
\end{theorem}
The proof of this theorem relies on a dichotomy within the class of all sparse degree sequences $\cD$ of length $n$.
This dichotomy is expressed through Theorems~\ref{thm:thres} and~\ref{thm:rob} below, which are much stronger, albeit somewhat technical, results.
Roughly speaking, if the tail of the degree sequence $\cD$ is sufficiently thin (conditions $A_1$ and $A_2$ below are satisfied), then there exists a critical probability $p_\cri$ bounded away from 0 
(essentially determined by $\cD$) such that when $p$ crosses $p_\cri$ the fraction of vertices that belong to
the largest component undergoes a rapid increase with high probability. 
On the other hand, if the tail is sufficiently heavy (either $A_1$ or $A_2$ are not satisfied), then for every $p\in (0,1]$, there is a giant component with high probability.

Let us now make these statements precise. 
For every $x>0$ and every $c\in \N$, 
we say that $\cD$ satisfies $A_1(x,c)$ if 
\begin{align} \label{eq:A_1}
\sum_{i\in W(c,\cD)} d_i&\leq \frac{x}{c}\cdot n \;.
\end{align}
For all $x>0$ and  $c_1,c_2\in \N$, 
we say that $\cD$ satisfies $A_2(x,c_1,c_2)$ if
\begin{align} \label{eq:A_2}
\sum_{i\in W(c_1,\cD)\sm W(c_2,\cD)}d_i^2 &\leq \frac{x}{4}\cdot n\;.
\end{align}

Note that being strongly uniformly integrable is equivalent to satisfying condition $A_1(\eps,c_{0})$ for every $\eps>0$ and $c_0= c_0(\eps)$. Also observe that these integrability notions naturally extend to $D_n$, even if the sequence $D_n$ does not converge in distribution.

The first part of this dichotomy describes which degree sequences have a percolation threshold.
\begin{theorem}\label{thm:thres}
For all $\epsilon, \gamma\in (0,1)$, all $c_1,c_2\in \mathbb{N}$ and $ \avd\geq 1$,
there exist $\rho=\rho(\epsilon,c_1)$, $\eta=\eta(\gamma,\epsilon,c_1)$ and $n_0$ such that 
for every $n\geq n_0$ and every degree sequence $\cD=(d_1,\dots,d_n)$ with average degree at most $\avd$ that satisfies $A_1(\eta,c_2)$, then for 
\begin{align}\label{eq:crit_prob}
p_\cri:=p_\cri (c_2,\cD) = \min \left\{ \frac{\sum_{i \in V\setminus W(c_2)} d_i}{\sum_{i \in V\setminus W(c_2)} d_i (d_i-1)},1\right\}
\end{align} 
we have
\begin{enumerate}[label=(\roman*)]
\item if $0\leq p\leq (1-\epsilon)p_\cri$, then
\begin{align*}
\Pro[L_1(G^\cD_p)>\gamma n]= o(1)\; ;
\end{align*}
\item if $\cD$ satisfies $A_2(\epsilon,c_1,c_2)$ and $(1+\epsilon)p_\cri\leq p \leq 1$, then
\begin{align*}
\Pro[L_1(G^\cD_p)>\rho n]= 1-o(1)\;.
\end{align*}
\end{enumerate}
\end{theorem}

In order to obtain a meaningful statement we need to apply Theorem~\ref{thm:thres} as follows: first, choose $\eps$ (width of the transition window) and $c_1,c_2,\avd$. This fixes the value of $p_\cri$ and $\rho$. Now, choose $\gamma$ which might be arbitrarily smaller than $\rho$. This fixes $\eta$ and $n_0$. After these choices, the theorem then gives a sufficient criterion for degree sequences whose size of the largest component jumps from at most $\gamma n$ to at least $\rho n$ in a window of width $2\eps$ around $p_\cri$.

Interestingly, this theorem gives a criterion for the existence of ``sudden'' jumps in $L_1(G_p^\cD)$ that do not necessarily correspond to the phase transition of the appearance of a linear order component. In particular, it applies to cases where the degree sequence is not uniformly integrable. In Section~\ref{sec:exm} we will see an example of this behaviour.
\medskip

The other part of the dichotomy settles the case of robust degree sequences.
\begin{theorem}\label{thm:rob}
For all $p,\delta \in (0,1)$ and all $\avd \geq 1$, 
there exist $K,c_0\in \N$ such that for every $c\geq c_0$, 
there exist $\rho>0$ and $n_0\in \N$ such that for every $n\geq n_0$ and every degree sequence $\cD=(d_1,\dots,d_n)$ with average degree at most $\avd$ that does not satisfy $A_1(K,c)$ or $A_2(K,0,c)$,
then
\begin{align*}
\Pro[L_1(G^\cD_p)>\rho n]\geq 1-\delta\;.
\end{align*}
\end{theorem}

As defined in~\eqref{eq:crit_prob}, we have $p_\cri>0$~(assuming that $0/0=1$). Intuitively speaking, Theorem~\ref{thm:rob} classifies the degree sequences that have a ``critical percolation threshold'' at $p=0$. It is also worth noticing that if $p=p(n)\to 0$ and $\avd=O(1)$, then $L_1(G_p^\cD)$ is at most the number of edges in $G_p^\cD$, which is $O(p\avd n)= o(n)$ with probability $1-o(1)$.

Theorems~\ref{thm:thres} and~\ref{thm:rob} show that the existence of a critical $p$ for the emergence  of a giant component is determined by the shape of the degree sequence for degrees that are bounded by a constant as well as by the degree sum of vertices of larger degree.  For example, whether a degree sequence contains one vertex of degree $n/2$ or $n/(2\log n)$ vertices of  degree $\log n$ does not make any difference.

\subsection{Approach to the proof of Theorems~\ref{thm:thres} and~\ref{thm:rob}}

Previous work on bond percolation in $G^\cD$ relies on the study of the configuration model. Given $\cD=(d_1,\dots,d_n)$, let $\hat{G}^\cD$ denote the random (multi)graph obtained using the configuration model (see e.g.~\cite{wormald1999models}). 
The first author observed that the percolated random graph $\hat{G}^\cD_p$ has the same distribution as $\hat{G}^{\cD_p}$
where $\cD_p=(d^p_1,\dots,d^p_n)$ is the random sequence obtained by choosing $d^p_i$ distributed as a binomial random variable with parameters $d_i$ and $p$, conditional on $\Sigma_{i\in[n]}d_i^p$ being even~(see Lemma 3.1 in~\cite{fountoulakis2007percolation}). Loosely speaking, this result states that one could interchange the two random processes, percolating first the degree sequence conditional on its sum being even) and then choosing a random graph with the percolated degree sequence. Using this observation one can transfer results for the configuration model to its percolated instances~\cite{bollobas2015old,fountoulakis2007percolation,janson2008percolation}.
These results can be transferred to the simple random graph $G_p^\cD$ provided $\cD$ satisfies certain technical conditions~(see Section~\ref{sec:previous}). 

Joos et al.~\cite{joos2016how} established a criterion for the existence of a linear order component for any degree sequence. Following the previous observation of interchanging the two random processes, one could hope that the largest component of $G_p^\cD$ could be studied directly, applying this criterion. 
However, the following example discusses a degree sequence for which the random graphs $G^\cD_p$ and $G^{\cD_p}$ are drastically different.

\begin{example}\label{exm:1}
Consider the degree sequence $\cD_n= (n/4,n/4, 1,1\dots,1)$ and let $p=1/2$.
By standard concentration inequalities, with high probability, the degree sequence $\cD_p$ satisfies $d^p_1, d^p_2=(1+o(1))n/8$ and $\Sigma_{i\in[n]}d_i^p = (1+o(1))3n/4$. The sequence $\cD_p$, if feasible, it is highly likely to satisfy $v_1v_2\in E(G^{\cD_p})$ (this can be shown by an easy switching argument; see Section~\ref{sec:switching}). However, by definition, the probability that $v_1v_2\in E(G^{\cD}_p)$ is at most $1/2$. As all vertices different from $v_1$ and $v_2$ have degree one, the order of the largest component will strongly depend on the existence of the edge $v_1v_2$. So the component structure of $G^\cD_p$ and $G^{\cD_p}$ is different.
\end{example}

\subsection{Comparison of Theorem~\ref{thm:seq} to previous results}\label{sec:previous}

The strongest statements in our paper are Theorems~\ref{thm:thres} and~\ref{thm:rob}. However, as previous results deal with sequences of degree sequences $\fD=(\cD_n)_{n\geq 2}$, it is more convenient to compare them to Theorem~\ref{thm:seq}. Additionally, we will assume that $D_n$ converges in distribution to the random variable $D$, denoted by $D_n\to D$, where $D$ has finite and positive mean $\bE[D]$; that is, there exists a probability distribution $(r_k)_{k\geq 0}$ such that for every $k\in \bN$
\begin{align}\label{eq:lim}
\lim _{n\to \infty} \frac{\{i\in [n]:\, d^{(n)}_i=k\}}{n} = r_k\;.
\end{align}
and $\bE[D] =\sum_{k\geq 0} k r_k \in (0,\infty)$.
Observe that~\eqref{eq:lim} implies that $d(\fD)$ exists.

Note that Theorem~\ref{thm:seq} only requires the existence of $d$. This is a slightly  weaker condition than the convergence of $D_n$ in distribution, but it is similar in spirit.

In this context, consider the following condition on the convergence of means
\begin{align}\label{eq:UI_lim}
\lim_{n\to \infty} \bE[D_n] =\bE[D]\;. 
\end{align}
Condition~\eqref{eq:UI_lim} can be easily replaced by the slightly weaker condition that $D$ is uniformly integrable (see e.g. Remark 2.2 in~\cite{janson2009new}). Moreover, given that $D_n\to D$, the condition that 
$D_n$ is (strongly) uniformly integrable as in~\eqref{eq:UI_seq} (and~\eqref{eq:SUI_seq}) is equivalent to $D$ being (strongly) uniformly integrable.

Additionally, we say that the sequence $(D_n)_{n\geq 2}$ 
has \emph{bounded second moment} if
\begin{align}\label{eq:BSO}
\bE[D_n^2]= O(1) \;. 
\end{align}
Under the assumption $D_n\to D$, we have following implications: 
\medskip

%
%

\begin{center}
\begin{tabular}{ccccc}
 $d(\fD)<\infty$& $\Leftarrow$ & \emph{Bounded second moment}& $\Rightarrow$ &\textit{(Strong) uniform integrability} 
\end{tabular}
\end{center}
\medskip
All the implications are straightforward to check, so we omit their proofs. 

We now state the two central results from the literature on bond percolation.
\begin{theorem}[Proposition 3.1 in~\cite{janson2008percolation}] \label{thm:jan} 
Suppose $\fD= (\cD_n)_{n\geq 2}$ is a sequence of degree sequences such that
$D_n \to D$ and $(D_n)_{n\geq 2}$ has bounded second moment.
If $p>1/d(\fD)$, there exists $\xi>0$ such that
$$
\frac{L_1(G_p^{\cD_n})}{n} \stackrel{p}{\to} \xi\;,
$$
and if $p<1/d(\fD)$, then
$$
\frac{L_1(G_p^{\cD_n})}{n} \stackrel{p}{\to} 0\;.
$$
\end{theorem}
Now,  let $\rho(D)$ be the survival probability of a Galton-Watson tree with offspring distribution given by $D$. Let $D_p$ be the \emph{$p$-thinned} version of $D$, defined by 
$$
\bP[D_p=i]=\sum_{j\geq i} \bP[D=j] \binom{j}{i}p^i(1-p)^{j-i}\;.
$$
\begin{theorem}[Theorem 22 in~\cite{bollobas2015old}]\label{thm:bol_rio} 
Suppose $\fD= (\cD_n)_{n\geq 2}$ is a sequence of degree sequences such that
$D_n \to D$ and $(D_n)_{n\geq 2}$ is uniformly integrable.
Then for every $p\in (0,1)$
$$
\frac{L_1(G_p^{\cD_n})}{n} \stackrel{p}{\to} \rho(D_p)\;.
$$
\end{theorem}
The aim of this paper is to obtain results on the existence of a linear order component in $G_{p}^\cD$ that are as widely applicable as possible, even if some precision is lost due to their generality. While Theorems~\ref{thm:jan} and~\ref{thm:bol_rio} are more precise in their conclusions, they require conditions on $\fD$ (bounded second moment and uniform integrability, respectively) that are not necessary in Theorem~\ref{thm:seq}. For instance, the results in Theorem~\ref{thm:seq}~\ref{item:2} when $\fD$ is not uniformly integrable are not implied by any of the previous results in the literature. As we will show in the next sub-section, the case of non-uniformly integrable sequences is particularly interesting, and one cannot hope for very strong results in this setting. In conclusion, Theorem~\ref{thm:seq} and the previous results complement each other, and their use is a compromise between precision and generality.
%


\subsection{Non-uniformly integrable sequences}\label{sec:exm}

Roughly speaking, a degree sequence is not uniformly integrable if the vertices of unbounded degree have non-negligible contribution to the average degree. Here, we include two illustrative examples.

\begin{example} The order of the largest component may not be concentrated.
Consider again the degree sequence $\cD_n= (n/4,n/4, 1,1\dots,1)$ given in Example~\ref{exm:1} and let $p\in (0,1)$. 
Using switchings, one can show that $v_1v_2\in E(G^{\cD_n})$ with probability $1-o(1)$.
Hence,
for any fixed $p\in (0,1)$ the probability that $v_1v_2\in E(G_p^{\cD_n})$ is $(1-o(1))p$, bounded away from $0$ and $1$. 
If $v_1,v_2$ are adjacent in $G_p^{\cD_n}$, the order of the largest component will be distributed as $\Bin(n/2-2, p)$; 
otherwise, it will be distributed as the maximum over two independent copies of $\Bin(n/4,p)$. 

This example can be generalised to produce degree sequences  $\cD_n$ satisfying the following: for every $\rho>0$, there exists $\delta>0$ such that 
\begin{align}
\bP[ L_1(G_p^{\cD_n}) < \rho n] &>\delta \label{eq:ASFS}\\
\bP[ L_1(G_p^{\cD_n}) > (1-\rho) n]  & > \delta\;. \nonumber
\end{align}
Thus, one cannot expect that $n^{-1} L_1(G_p^{\cD_n})$ converges in probability  to a constant as in Theorems~\ref{thm:jan} and~\ref{thm:bol_rio}. Moreover,~\eqref{eq:ASFS} shows that the statement of Theorem~\ref{thm:seq} $(ii)$ cannot hold with probability $1-o(1)$.
\end{example}

\begin{example}
As mentioned above, Theorem~\ref{thm:thres} can be used to detect sudden changes of the order of the largest component that do not coincide with the birth of the giant component. Consider the following degree sequence $\cD_n=(\xi n, 3,3,\dots, 3)$, for some $\xi>0$. 
Theorem~\ref{thm:rob} shows that the size of the largest component of $G_p^{\cD_n}$ is linear for every $p>0$, in fact, it stochastically dominates a $\Bin(\xi n,p)$. However, we can also apply Theorem~\ref{thm:thres} with  $p_{\text{crit}}=1/2$.
In particular, given $\eps,\gamma,\rho$, we can choose $\xi$ sufficiently small so Theorem~\ref{thm:thres} gives a jump in $L_1(G_p^{\cD_n})$ from at most $\gamma n$ to at least $\rho n$ in a window of width $2\eps$ around $p_{\text{crit}}$.  Hence the birth of the giant component is at $p=0$ (by Theorem~\ref{thm:rob}) and there is a boost at $p=1/2$ (by Theorem~\ref{thm:thres}). 

Intuitively, if a sequence is not uniformly integrable, then it has linearly many edges in vertices of unbounded degree. Under some weak conditions, these vertices will typically form a connected core, even after percolation. This core contains linearly many edges and it typically creates a linear order component. However, if $p_{\text{crit}}>0$, then for $p> p_\cri$, the vertices of bounded degree will percolate even without the help of unbounded degree vertices, and thus, the growth of the giant component changes at $p_{\text{crit}}$. To our knowledge, this critical point has not been studied in the literature.  It would be interesting to get a better understanding of the size of the largest component around this point.
\end{example}

%

\medskip

\textbf{Structure of the paper:} The paper is structured as follows. In Section~\ref{sec:not}, we provide the basic notation and some technical estimates that will be used throughout the proof. Section~\ref{sec:overview} presents the main combinatorial tool we will use, the switching method, and provides an overview of the proof of Theorem~\ref{thm:thres} and~\ref{thm:rob}. In Section~\ref{sec:3_tec}, we present three important technical propositions. Assuming them, in Section~\ref{sec:thres} and~\ref{sec:rob} we prove Theorem~\ref{thm:thres} and~\ref{thm:rob}, respectively. 
In Section~\ref{sec:det_lem},~\ref{sec:explo} and~\ref{sec:very_rob} we prove these three propositions. 
Section~\ref{sec:sequ} is devoted to the proof of Theorem~\ref{thm:seq}. We provide an application of the results obtained to power-law degree sequences in Section~\ref{sec:PL}.
Finally, in Section~\ref{sec:remarks}, we state some remarks of our results and discuss a number of open questions.

\section{Notation and some probabilistic tools}\label{sec:not}

We consider labelled graphs $G$ with vertex set $V=V(G)=[n]:=\{1,\dots,n\}$ and edge set $E(G)$.
If we refer to a graph or degree sequence on the set $V$, 
we always implicitly assume that $V=[n]$ and thus $|V|=n$.
We say that a graph $G$ on $V$ 
has degree sequence $\cD=(d_1,\dots, d_n)$ if for every $i\in [n]$, the degree of $i$ is $d_i$. 
We let $\Td$ denote the sum of these degrees; that is, $\Td := \sum_{i=1}^n d_i$.
We denote by $|\cD|$ the length of the degree sequence. 

For an arbitrary vertex 
$v\in V$, we will often write $d(v)=d_G(v)$ for its degree. 
Let $H$ be a subgraph of $G$; if $v \in V(H)$, then $d_{H}(v)$ 
denotes the degree of $v$ in $H$; 
if $v \in V\sm V(H)$, then $d_{H}(v)=0$. 
For a graph $G$, a subset of vertices $U\subseteq V$ and $v\in V$, 
we occasionally use the notation $d_G(v,U)$ to denote the number of neighbours of $v$ in $G$ that are in the subset $U$. 
Given a degree sequence $\cD$ and a graph $G$, we denote by $\Delta(\cD)$ (and $\delta(\cD)$) and by $\Delta(G)$ (and $\delta(G)$) the maximum (and minimum) degree of the sequence $\cD$ and of the graph $G$, respectively.

We denote by $N(v)=N_G(v)$ the set of neighbours of $v$ in $G$.
For $S\subseteq V$, we use $N(S)=N_G(S)$ for the set of vertices in $V\sm S$ that have a neighbour in $S$. 
We also use $N[S]=N_G[S]$ for the set of vertices that are either in $S$ or in $N(S)$.
For $S\subseteq V$, we denote by $G[S]$ the (sub)graph of $G$ induced by $S$. 
For disjoint $S,T\subseteq  V$, we denote by $G[S,T]$ the bipartite graph induced between $S$ and $T$.

We will make use of some classical concentration inequalities that can be found in~\cite{molloy2013graph}.
\begin{lemma}[Chernoff's inequality] \label{lem:Chernoff}
Let $X_1, \dots, X_N$ be a set of independent Bernoulli random variables with expected value $p$ and let $X=\sum_{i=1}^N X_i$.  Then, for every $0<t<Np$
$$
\Pro[|X-\bE[X]|>t]\leq 2e^{-\frac{t^2}{3N p}}\;.
$$
\end{lemma}

\begin{lemma}[McDiarmid's inequality] \label{lem:McDineq}
Let $X_1, \dots, X_s$ be a set of independent random variables taking values in $[0,1]$. Let $f: [0,1]^s\to \mathbb{R}$ be a function of $X_1,\dots, X_s$ that satisfies for every $1\leq i \leq s$, every $x_1,\dots,x_s \in [0,1]$ and every $x_i' \in [0,1]$,
$$
|f(x_1,\dots, x_i, \dots, x_s)- f(x_1,\dots, x'_i, \dots, x_s)|\leq c_i\;,
$$
for some $c_i>0$. Then, for every $t>0$
$$
\Pro[|f(X_1,\dots, X_s)-\bE[f(X_1,\dots, X_s)]|>t]\leq 2e^{-\frac{2t^2}{\sum_{i=1}^s c_i^2}}\;.
$$
\end{lemma}

Many of our results have complicated hierarchies of constants.
To be precise, 
if we say that a statement holds whenever $a \ll b \ll c \leq 1$, 
then this means that there are non-decreasing
functions $f,g : (0, 1] \rightarrow (0, 1]$ such that the result holds for all $0 < a, b,c \leq 1 $ with $a \leq f(b)$ and $b \leq g(c)$.
In particular, such hierarchies need to be read from right to left.
We will not calculate these functions explicitly in order to simplify the presentation, 
but one could easily calculate them from inspecting our proofs.
Hierarchies with more terms are defined in a similar way.
Finally, we write $x = a\pm b$ to denote that $x\in [a-b,a+b]$, for real numbers $x,a,b$.

\section{Overview of the proofs}\label{sec:overview}
In this section we present an overview of the proofs of Theorems~\ref{thm:thres} and~\ref{thm:rob}
as well as of the main method used in them.
\subsection{The switching method}\label{sec:switching}
Let $\cG^\cD$ be the set of simple graphs with degree sequence $\cD$. Throughout our proofs we want to 
consider the probability that $G^\cD$, a uniformly chosen element of $\cG^\cD$, satisfies a certain property. 
We will slightly abuse notation by indistinctly referring to $\cG^\cD$ as a set of graphs and as a probability space 
equipped with the uniform distribution. Similarly, we will consider $\cF\subseteq \cG^\cD$ and talk about the 
probability of $\cF$, thought as an event in the probability space $\cG^\cD$.

The main combinatorial tool that we use to estimate different probabilities in $\cG^\cD$ is the \emph{switching method}. 
Given a graph $G$ with degree sequence $\cD$ and two ordered edges $uv,xy\in E(G)$ 
we can perform the following graph 
operation, called a $\{uv,xy\}$-switch, or simply a \emph{switch}: obtain $G'$ by deleting the edges $uv$ and $xy$ from 
$G$ and adding the edges $ux$ and $vy$ in $G$. Observe that the $\{uv,xy\}$-switch is different from the $\{vu,xy\}$-switch, but equal to the $\{vu,yx\}$-switch.
If either $ux$ or $vy$ are edges of $G$, 
the graph $G'$ will have multiple edges, and if either $u=x$ or $v=y$, 
the graph $G'$ will have loops. Since we restrict here to simple graphs, we say that a switch is \emph{valid} if none of these occur.

The basic idea of the switching method is the following.
In order to determine the probability of $\cF\subseteq \cG^\cD$ in terms of 
the probability of $\cF'\subseteq \cG^\cD$, 
we use the average number of valid switches
between a graph in $\cF$ and a graph in $\cF'$, denoted by $d(\cF\to\cF')$, and vice versa. A simple double-counting of such 
switches gives
\begin{align} \label{eq:fund}
\Pro[\cF]= \frac{d(\cF'\to\cF)}{d(\cF\to\cF')}\cdot \Pro[\cF']\;.
\end{align}
Although this relation is very simple,
the switching method is very powerful.
In particular, we avoid the use of the configuration model and all the technicalities that come with it.

\subsection{The emergence of the giant component}

In~\cite{joos2016how}, a characterisation is given of those degree sequences $\cD$ for which the random graph $G^{\cD}$ has 
a giant component with high probability. 
Let $j^\cD$ be the smallest $j\in \N$ such that 
$$\sum_{i=1}^j d_i( d_i-2)>0$$
if such $j$ exists and else $j^\cD=n$.
Also, they set $R^{\cD}:= \sum_{i=j^{\cD}}^n d_i$ and $M^{\cD}:= \sum_{i:d_i\neq 2} d_i$.  
Effectively, $G^{\cD}$ has a giant component with high probability if and only if $R^{\cD}$ grows linearly in $M^\cD$. 

As we deal with bond percolation on $G_p^{\cD}$, 
we need to consider 
generalizations
of these quantities. 
\begin{enumerate}[label=(\roman*)]
\item Let $j^\cD_p$ be the minimum between $n$ and the smallest natural number $j$ such that 
$$\sum_{i=1}^j d_i(p(d_i-1)-1)>0\;.$$
\item Let $R^\cD_p := \sum_{i=j^\cD_p}^n (p(d_i-1)-1)$.
\end{enumerate}
Observe that if $\sum_{i=1}^n d_i(p(d_i-1)-1)\leq 0$, we have $R^\cD_p = p(d_n-1)-1$.  

Note that in the case $p=1$, the definition of  $j^\cD_p$ coincides with the definition of $j^\cD$.
Moreover, $R^\cD_1\leq R^\cD \leq 3 R^\cD_1$. In our setting we do not need to define an analogue of $M^\cD$ since vertices of degree $2$ play no special role here (see Section~\ref{sec:remarks} for a detailed discussion).
Also note that vertices of degree $0$ have no contribution to 
these parameters. 
So, it is useful to assume, and we will do so if not stated otherwise, that $d_i\geq 1$ for every $i\in [n]$.
The result for degree sequences containing vertices of degree $0$ can be easily deduced from the analysis of degree sequences without them.


Theorem~\ref{thm:thres} and~\ref{thm:rob} essentially distinguish between two cases on the tail of the degree sequence. 
In Theorem~\ref{thm:thres}, we show that a critical percolation threshold $p_\cri$ exists if the two conditions $A_1(\cdot,c_1,c_2)$ and $A_2(\cdot,c_2)$ hold. 
These conditions bound the number of edges incident to vertices of large degree. One should note that the two conditions are only required for particular values of the degrees, namely $c_1$ and $c_2$, and thus, they are much weaker than a domination condition on the whole tail of the degree sequence.
The heart of the proof of Theorem~\ref{thm:thres} is the analysis of  an \emph{exploration process}, 
in which we reveal the components 
of $G_p^{\cD}$ by exposing the neighbours of each vertex sequentially (cf. Proposition~\ref{prop:explo}).  
To avoid technical difficulties that arise due to high degree vertices, we include them together with their neighbours in the 
initial set of explored vertices. So during the process, we only reveal the connections of vertices of low or moderately-growing degree. 
Let $S$ denote the set of high degree vertices (we will specify the exact magnitude during the proof). 
We show that if $\sum_{v \not \in N[S]} d(v) (p(d(v)-1)-1)$ is negative and actually decays linearly with $n$, then the exploration process
 is subcritical in the sense that all components it reveals are sub-linear. 
If $\sum_{v \not \in N[S]} d(v) (p(d(v)-1)-1)$ is positive and grows linearly,  then one can use condition $A_2$ to ensure that 
 $R^{\cD}_p$ grows linearly. In that case, it turns out that the exploration process will reveal a component of linear order with high probability. 
In Sections~\ref{sec:thres} and~\ref{sec:rob}, we show how these two conditions give a critical value $p_\cri$ such that when $p$ goes from less than $(1- \epsilon)p_\cri$, 
to at least $(1+\epsilon) p_\cri$, 
the fraction of vertices in the largest component undergoes an abrupt increase.

%
 
Regarding Theorem~\ref{thm:rob}, recall that its premises cover degree sequences that have a quite heavy tail, that is, either Condition $A_1$
or Condition $A_2$ fails. Here, we distinguish between two sub-cases. The first is that a set $S_1$ of very high (growing) degree vertices have linear 
total degree. The boundedness of the average degree implies that $S_1$ is small (of sub-linear size). Moreover, we show that with
high probability $G_p^\cD[S_1]$ is connected and that more than half of the edges incident to $S_1$ in $G^\cD$, have their other endpoint in $V\sm S_1$. Deleting each such
edge with probability that is bounded away from 0 leaves with high probability a giant component. This is stated in Proposition~\ref{prop:very_rob}.
Now, if $S_1$ does not have linear total degree, then we show that its removal leaves a degree sequence $\cD'$ that is \emph{super-critical}: the quantity $R^{\cD'}_p$ grows linearly in $n$. One can then apply again Proposition~\ref{prop:explo} to find a giant component, concluding the proof of Theorem~\ref{thm:rob}. 

The transition from $R^{\cD}_p$ to $R^{\cD'}_p$ requires a result (Proposition~\ref{prop:imp}) which shows that if 
$R^{\cD}_p$ grows linearly in $n$ and we remove a set of vertices of small total degree, then the resulting degree 
sequence $\cD'$ is such that $R^{\cD'}_p$ still grows linearly, albeit with a smaller coefficient.

Finally, the proof of Theorem~\ref{thm:seq} is a relatively straightforward application of both Theorem~\ref{thm:thres} and~\ref{thm:rob} and it is proved in Section~\ref{sec:sequ}.

\section{Three technical propositions}\label{sec:3_tec}

In this section we introduce three important propositions that will allow us to prove Theorem~\ref{thm:thres} and~\ref{thm:rob}. We defer their proofs to Sections~\ref{sec:det_lem},~\ref{sec:explo} and~\ref{sec:very_rob}, respectively.

The first one is a \emph{deterministic} proposition which proves the following. 
Suppose $G$ is a graph with degree sequence $\cD$ and $S\subseteq V(G)$ such that $\sum_{u\in S}d(u)$ is small.
Suppose $\cD'$ is a possible degree sequence of the graph $G-S$,
then $R^{\cD'}_p$ is bounded from below by $R^{\cD}_p/50$.


\begin{proposition}\label{prop:imp}
Suppose $1/n\ll\nu$, $400\nu\leq\mu\leq 1$, and $p\in (0,1]$. 
Suppose $\cD$ is a degree sequence on $V$ with $R^\cD_p\geq \mu n$ and let $S\subseteq V$ be 
such that $\sum_{v\in S}d(v)\leq \nu n$. 
Assume that  $G'$ is a graph obtained from a graph $G$ with degree sequence $\cD$ by deleting all vertices in $S$ and afterwards by deleting all vertices of degree~$0$. Let
$\cD'$ be the degree sequence of $G'$ and assume it has length $n'$. 
Then $n'\geq (1-2\nu)n$ and $R^{\cD'}_p\geq \frac{\mu}{50} n'$. 

Moreover, if $G=G^\cD$ and $\cD'$ is the degree sequence of $G'$, then $G'$ is a uniformly random graph with degree sequence $\cD'$, that is $G'=G^{\cD'}$.
\end{proposition}

The key ingredient for the proof of both Theorem~\ref{thm:thres} and~\ref{thm:rob} is the following proposition 
that gives us the component structure of $G^\cD_p$. 
This proposition can be viewed as the extension of the main theorem in~\cite{joos2016how}. 
It states that if the drift on the bulk of the vertices (that is, excluding vertices of very high degree and their neighbours) 
is negative, then the fraction of vertices in the largest component is a.a.s.~bounded by some small constant. 
On the other hand,
if $\Delta(\cD)\leq n^{1/4}$ and $R_p^{\cD}\geq \mu n$, then we have a.a.s.~a component containing a constant fraction of all vertices.

\begin{proposition}\label{prop:explo}
Suppose $n\in \N$ and $1/n\ll \alpha\ll \gamma\ll\mu,1/\avd,p\leq 1$. 
Let $\cD$ be a degree sequence on $V$ with $\Td\leq \avd n$.
\begin{enumerate}[label=(\roman*)]
\item If there exists a set $S\subseteq V$ such that $d(v)\leq n^{1/4}$ for every $v\notin S$,  and for every graph $G$ with degree sequence $\cD$ one has
$\sum_{u\in N_G[S]}d(u)\leq \alpha n$, 
and $\sum_{u\in V\sm N_G[S]}d(u)(p(d(u)-1)-1)\leq -\mu n$, then
$$
\Pro[L_1(G^\cD_p)\leq \gamma n]=1-o(1)\;.
$$
\item If $\Delta(\cD)\leq n^{1/4}$ and $R_p^{\cD}\geq \mu n$, then
$$
\Pro[L_1(G^\cD_p)\geq \gamma n]=1-o(1)\;.
$$
\end{enumerate}
\end{proposition}


Our last result is a version of Theorem~\ref{thm:rob} for degree sequences that have many edges incident to vertices of unbounded degree. Recall that for $c\in\N$, $W(c,\cD)$ denotes the set of vertices of degree at least 
$c$. 
\begin{proposition}\label{prop:very_rob}
For all $p,\delta,\epsilon\in (0,1)$ and all $\avd\geq 1$, 
there exist $\gamma>0, n_0\geq 1$ such that for every $n\geq n_0$ and every degree sequence $\cD$ on $V$
with $\Td\leq \avd n$ that satisfies
$$
\sum_{i\in W(\log^2 n, \cD)} d_i\geq \epsilon n\;,
$$
we have
$$
\Pro[L_1(G^\cD_p)> \gamma n]\geq 1-\delta\;.
$$
\end{proposition}

In the next two sections we proceed with the proof of Theorem~\ref{thm:thres} and~\ref{thm:rob} assuming these three propositions.

\section{Degree sequences with thin tails: proof of Theorem~\ref{thm:thres}} \label{sec:thres}

\begin{proof}[Proof of Theorem~\ref{thm:thres}] 
Let $\cD$ be a degree sequence on $V$.
Recall that $W(c,\cD)=\{i \in V: d_i \geq  c\}$; for convenience, we denote $W(c):=W(c,\cD)$.
We assume that $\eta$ is small enough with respect to $\epsilon,\gamma,1/c_1 $ so that the inequalities 
in this proof hold. 
Next, let $p_\cri:=p_\cri (c_2,\cD)$ be as in~\eqref{eq:crit_prob}.
Note that the definition of $p_\cri$ excludes the contribution of all the vertices of degree at least $c_2$. Moreover, with this definition we have
\begin{align}\label{eq:impli}
\sum_{i\in V\setminus W(c_2)} d_i (p_\cri( d_i-1)-1)\leq 0
\end{align} 
and equality holds if $p_\cri<1$.

We first prove Theorem~\ref{thm:thres}~(i). 
Suppose that $p \leq (1-\epsilon)p_\cri$.
Our strategy is to apply Proposition~\ref{prop:explo}~(i) with $W(c_2)$, $2\eta$ and 
$\epsilon/3$ playing the role of $S$, $\alpha$ and $\mu$, respectively.  
In order to do so, we need to give an upper bound on $\sum_{i\in N[W(c_2)]}d_i$ and on $\sum_{i\in V\sm N[W(c_2)]}d_i(p(d_i-1)-1)$ for every graph $G$ with degree sequence $\cD$.

By assumption, $A_1(\eta,c_2)$ holds; that is,
\begin{equation}\label{eq:ubN[S]}
|N(W(c_2))| \leq \sum_{i\in W(c_2)} d_i\leq\frac{ \eta }{c_2}\cdot n,
\end{equation} 
and therefore
\begin{align}
\sum_{i\in N[W(c_2)]} d_i
&\leq \sum_{i\in W(c_2)} d_i+\sum_{i\in N(W(c_2))} d_i  
\leq \frac{\eta}{c_2}\cdot n + c_2 |N(W(c_2))| 
\leq  (1+c_2)\frac{\eta}{c_2}\cdot n 
\leq 2\eta n\;.\label{eq:ubN[S]2}
\end{align} 

We next bound $\sum_{i\in V\sm N[W(c_2)]} d_i(p(d_i-1)-1)$ from above. 
Since $d_i(p(d_i-1)-1)\geq -d_i$ for every $i\in V$, we obtain
\begin{eqnarray}\label{eq:neigh}
\sum_{i\in V\sm N[W(c_2)]} \!\!\!\! d_i(p(d_i-1)-1) 
&=&\!\!\!\!\sum_{i\in V\sm W(c_2)}\!\!\!\! d_i(p(d_i-1)-1) - \!\!\!\!\sum_{i\in N(W(c_2))} \!\!\!\! d_i(p(d_i-1)-1)\nonumber \\
&\leq &\!\!\!\!\sum_{i\in V\sm W(c_2)}\!\!\!\! d_i(p(d_i-1)-1) + c_2|N(W(c_2))|\nonumber \\
& \stackrel{(\ref{eq:ubN[S]})}{\leq}&\!\!\!\!\sum_{i\in V\sm W(c_2)}\!\!\!\! d_i(p(d_i-1)-1) +\eta n\;.
\end{eqnarray}

It follows that
\begin{align*}
\sum_{i\in V \sm W(c_2)} d_i(p(d_i-1)-1) 
&\leq  (1-\epsilon )\sum_{i\in V \sm  W(c_2)} d_i(p_\cri(d_i-1)-1) - \epsilon\sum_{i\in V\sm W(c_2)} d_i \\
&\stackrel{(\ref{eq:impli})}{\leq } - \epsilon\sum_{i\in V\sm W(c_2)} d_i\;.
\end{align*}
Using that $d_i\geq 1$, we obtain
\begin{equation} \label{eq:NWS} 
\sum_{i\in V\sm W(c_2)} d_i 
= \sum_{i\in V} d_i - \sum_{i\in W(c_2)} d_i 
\geq  n - \sum_{i\in W(c_2)} d_i 
 \stackrel{(\ref{eq:ubN[S]})}{\geq}  \frac{n}{2}.
\end{equation}
Therefore, 
\begin{align*}
\sum_{i\in V \sm W(c_2)} d_i(p(d_i-1)-1)
&\leq  -\frac{\epsilon }{2}\cdot n. 
\end{align*}
Using~\eqref{eq:neigh},  it follows that
\begin{align}
\sum_{i\in V \sm N[W(c_2)]} d_i(p(d_i-1)-1)
&\leq  -\frac{\epsilon }{2}\cdot n + \eta   n \leq -\frac{\epsilon}{3}\cdot n \label{eq:neq_drift4}
\end{align}
As \eqref{eq:ubN[S]2} and \eqref{eq:neq_drift4} hold, 
Proposition~\ref{prop:explo}~(i) completes the proof of Theorem~\ref{thm:thres}~(i).
\medskip

We proceed with the proof of Theorem~\ref{thm:thres}~(ii). 
Suppose now that $p\geq (1+\epsilon )p_\cri$. 
Here we may assume that $p_\cri < 1$ as otherwise $p>1$ does not satisfy the assumption of part~(ii).  
We will first show that there exists $\mu = \mu (\epsilon,\avd,c_1)$ such that $R^{\cD}_p \geq \mu n$.

For $k\in\{1,2\}$, we define
$j_k:=\min\{n+1, i\in [n]:d_i\geq c_k\}$. Since $p_\cri<1$, \eqref{eq:impli} holds with equality. 
Using the definition of $j^\cD_p$, we obtain
\begin{eqnarray}\notag
\sum_{i=j^\cD_p}^{ j_2-1} d_i(p(d_i-1)-1) 
&=& \sum_{i=1}^{ j_2-1} d_i(p(d_i-1)-1) -\sum_{i=1}^{j^\cD_p-1} d_i(p(d_i-1)-1)  \\
&\geq &\sum_{i=1}^{ j_2-1} d_i(p(d_i-1)-1) \notag\\ \notag
 &\geq &(1+\epsilon)\sum_{i=1}^{ j_2-1} d_i(p_\cri(d_i-1)-1)+\epsilon \sum_{i=1}^{ j_2-1} d_i \\ 
&\stackrel{\eqref{eq:impli}}{=} &0+  \epsilon\sum_{i=1}^{ j_2-1} d_i
\stackrel{\eqref{eq:NWS}}{\geq} \frac{\epsilon}{2}\cdot n \;. \label{eq:jj1}
\end{eqnarray}
It follows from $A_2(\epsilon,c_1,c_2)$ that
\begin{align*}
\sum_{i=j^\cD_p}^{j_1-1} d_i(p(d_i-1)-1) 
\stackrel{(\ref{eq:jj1})}{\geq} \frac{\epsilon}{2}\cdot n - \sum_{i=j_1}^{j_2-1} d_i(p(d_i-1)-1)
\geq  \frac{\epsilon}{2}\cdot n- \sum_{i=j_1}^{j_2-1} d_i^2
\geq \frac{\epsilon}{4}\cdot n.
\end{align*}
We conclude that 
$$
R^{\cD}_p = \sum_{i=j^{\cD}_p}^n (p(d_i-1)-1) 
\geq \sum_{i=j^\cD_p}^{ j_1-1} (p(d_i-1)-1)
\geq \frac{1}{c_1} \sum_{i=j^\cD_p}^{ j_1-1} d_i(p(d_i-1)-1) 
\geq \frac{\epsilon}{4 c_1}\cdot n =: \mu n. 
$$

Recall that, by assumption, 
we have $\sum_{i\in W(c_2)} d_i\leq \frac{\eta}{c_2}\cdot n\leq  \eta n$. 
Since we can choose $\eta$ small enough in terms of $\mu$, we can apply Proposition~\ref{prop:imp} with $\eta$ and $W(c_2)$ playing the role of $\nu$ and $S$.
Let $\cD'$ be the random degree sequence of the subgraph $G^\cD[V\sm W(c_2)]$. 
Let $\fD_0$ be the set of degree sequences $\cD_0$  that satisfy $\Pro[\cD'=\cD_0]>0$.
By Proposition~\ref{prop:imp}, we deduce that if $\cD_0\in \fD_0$, then
$R^{\cD_0}_p \geq \mu n/50$ and $|\cD_0|\geq \frac{n}{2}$. 
Moreover, conditional on $\cD'=\cD_0$, we have that $G^\cD[V\sm W(c_2)]$ and $G^{\cD_0}$ have the same probability distribution.
Now we select $\rho$ such that $\eta \ll\rho\ll \mu,1/\avd,p$. 
We can apply Proposition~\ref{prop:explo}~(ii) to the degree sequence $\cD_0$ with $2\rho$ playing the role of $\gamma$, from which Theorem~\ref{thm:thres}~(ii) follows:
\begin{align*}
\Pro[L_1(G^\cD_p)> \rho n]
&\geq \Pro[L_1(G^\cD_p[V\sm W(c_2)])> \rho n] \\
&\geq  \min_{\cD_0\in \fD_0} \Pro[L_1(G_p^{\cD_0}))> 2\rho |\cD_0|] \\
&= 1-\new{o}(1)\;.
\end{align*}
\end{proof}

\section{Robust degree sequences: proof of Theorem~\ref{thm:rob}} \label{sec:rob}


\begin{proof}[Proof of Theorem~\ref{thm:rob}]
Let  $c_0$ and $K$ be such that $1/c_0 \ll 1/K \ll \delta,p,\epsilon, 1/\avd$.
For any given $c\geq c_0$, assume that $n_0$ is such that $1/n_0\ll 1/c$.
Let $n\geq n_0$.

Suppose first that $\sum_{i\in W(\log^2 n)}d_i\geq Kn/c^3$. 
We apply Proposition~\ref{prop:very_rob} with $K/c^3$ playing the role of $\epsilon$ and obtain a $\gamma_1$ such that 
$\Pro[L_1(G^\cD_p)> \gamma_1 n]\geq 1-\delta$.

\medskip

Hence we may assume that $\sum_{i\in W(\log^2 n)}d_i\leq K n/c^3$. 
Let $S:= W(\log^2 n)$. 
We will show that the (random) subgraph $G^\cD_p[V\sm S]$ has a giant component, and thus also $G^\cD_p$ has a giant component. 
Let us first show that $R^\cD_p\geq\frac{Kp}{16c}\cdot n$. We consider two cases:
 
\medskip
\noindent
\textbf{Case 1:} $A_1\left(K,c\right)$ does not hold; that is, $\sum_{i\in W(c)}d_i\geq Kn/c$.
\smallskip

We define $j_1:=\min\{j\in[n]:d_j\geq c\}$ and let 
$j_2$ be the smallest integer $j$ such that $\sum_{i=j_1}^j d_i\geq Kn/(2c)$. 
Since $A_1\left(K,c\right)$ does not hold, $j_1$ and $j_2$ are well-defined. We have
 \begin{align*}
  \sum_{i=1}^{j_2} d_i(p(d_i-1)-1)
  &=p \sum_{i=1}^{j_2} d_i^2 -(1+p)\sum_{i=1}^{j_2}d_i
  \geq p\sum_{i=j_1}^{j_2}d_i^2 - (1+p)\avd n\\
  &\geq cp \sum_{i=j_1}^{j_2} d_i-2\avd n
  \geq \left(\frac{Kp}{2}-2\avd\right)n\;.
 \end{align*}
Therefore, $j^{\cD}_p\leq j_2$.
Since $1/{d_{j_2}}\leq 1/c\leq 1/c_0\ll p$,
we conclude that $p(d_{j}-1)-1\geq  pd_{j}/4$ for all $j\geq j_2$.
By the definition of $j_2$ and using $A_1(K,c)$ (in fact, its negation), it follows that
 $$
 R^\cD_p \geq \sum_{i=j_2}^{n} (p(d_i-1)-1)\geq \frac{p}{4}\sum_{i=j_2}^{n} d_i
\geq  \frac{p}{4}\left(\sum_{i\in W(c)} d_i- \sum_{i=j_1}^{j_2-1} d_i  \right)
\geq\frac{Kp}{16c}\cdot n
\;.
 $$

\medskip
\noindent
\textbf{Case 2:} $A_2\left(K,0,c\right)$ does not hold; that is, $\sum_{i\in V\sm W(c)} d_i^2 \geq {K}n/{4}$.
\smallskip

Now let $j_3$ be the smallest integer $j$ such that $\sum_{i=1}^j d_{i}^2\geq {K}n/{8}$. 
Since $A_2\left(K,0,c\right)$ does not hold, $j_3$ is well-defined and $d_{j_3}< c$. 
Using the definition of $j^\cD_p$, similarly as before
 \begin{align*}
 \sum_{j=j^\cD_p}^{j_3} d_j(p(d_j-1)-1)
 &\geq  \sum_{j=1}^{j_3} d_j(p(d_j-1)-1)\\
  &\geq p\sum_{j=1}^{j_3} d_j^2 - (1+p)\avd n\\
  &\geq \left(\frac{Kp}{8}-2\avd\right)n
  >\frac{Kp}{16}\cdot n\;.
 \end{align*}
Thus  
 $$
 R^\cD_p 
\geq \sum_{j=j^\cD_p}^{j_3} (p(d_j-1)-1)
> \frac{1}{c} \sum_{j=j^\cD_p}^{j_3} d_j(p(d_j-1)-1)
\geq  \frac{Kp}{16c}\cdot n
\;.
 $$
 
%

Let $\cD'$ be the random degree sequence of the subgraph $G^\cD[V\sm S]$. Let $\fD_0$ be the set of degree sequences $\cD_0$  that satisfy $\Pro[\cD'=\cD_0]>0$.
By Proposition~\ref{prop:imp} applied to $S$ with $K/c^3$ and $Kpn/(16c)$ playing the role of $\nu$ and $\mu$ (note that $400\nu \leq \mu$), 
we obtain that, for every $\cD_0\in \fD_0$, one has $R_p^{\cD_0}\geq  {Kpn}/{(800c)}$ and $|\cD_0|\geq \frac{n}{2}$.  
Moreover, conditional on $\cD'=\cD_0$, we have that $G^\cD[V\sm S]$ and $G^{\cD_0}$ have the same probability distribution. 
Using that  for every $\cD_0\in \fD_0$,  $\Delta(\cD_0)\leq \log^2{n}$ and $R_p^{\cD_0}\geq {Kpn}/{(800c)}$, we can apply Proposition~\ref{prop:explo}~(ii) and obtain $\gamma_2>0$
such that
\begin{align*}
\Pro[L_1(G^\cD_p)> \gamma_2 n]
&\geq \Pro[L_1(G^\cD_p[V\sm S])> \gamma_2 n] \\
&\geq  \min_{\cD_0\in \fD_0} \Pro[L_1(G_p^{\cD_0}))> 2\gamma_2 |\cD_0|] \\
&\geq 1-\new{o}(1)\;.
\end{align*}
We conclude the proof of the theorem by setting $\gamma:=\min\{\gamma_1,\gamma_2\}$.
\end{proof}

\section{Proof of Proposition~\ref{prop:imp}}\label{sec:det_lem}

Although the statement of Proposition~\ref{prop:imp} may sound very natural and also easy to prove,
the fact that some edge deletions may cause significant reordering in ordered degree sequences
makes the proof technical and complex.

\begin{proof}[Proof of Proposition~\ref{prop:imp}]

We start with the last part of the proposition. By exposing $\cD'$, we only fix the degree of each vertex in $V\setminus S$ towards $S$. Take any realisation $G'$ of a graph on $V\setminus S$ with degree sequence $\cD'$. The number of extensions of $G'$ to a graph on $V$ with degree sequence $\cD$ is clearly independent from the structure of $G'$. That is, for any choice of $G'$ there is the same number of extensions, and the law of $G^{\cD}[V\setminus S]$ coincides with that of $G^{\cD'}$.

It remains to prove the main part. For every $k\in [n]$,
let $d_k'$ be the degree of the vertex $k$ in $G'$
and define $r_k:=d_k-d_k'$.
Note that $r_k=d_k$ for every $k\in S$ whereby
$$
\sum_{k\in V} r_k\leq 2\nu n\;.
$$
Clearly, by deleting at most $\sum_{k\in S}d_k\leq  \nu n$ edges,
we have created at most $2 \nu n$ vertices of degree $0$, which we do not consider in  $\cD'$.
Let $n'$ be the number of vertices with positive degree after the deletion of $S$.
Thus $n'\geq (1-2\nu)n$.

Note that the statement follows if $\Delta(G)\geq \mu n/(40p)$, since then a vertex of maximum degree is not contained in $S$ and has degree at least $\Delta(G)-\nu n$ in $\cD'$, whereby
$$
R_p^{\cD'}\geq p((\Delta(G)-\nu n)-1)-1 \geq \frac{\mu }{50}\cdot n\geq \frac{\mu n'}{50}\;,
$$
where we used that $400\nu \leq \mu$. 
Thus, we may now assume that $\Delta(G) <  {\mu n}/(40p)$.

We define $f_k:=p(d_k-1)-1$ and $f_k':=p(d_k'-1)-1$ for all $i\in [n]$.
Recall that $d_1\leq \dots \leq d_n$ and that $j^\cD_p$ is the smallest integer $j$ such that
$$
\sum_{k=1}^{j} d_k ( p(d_k - 1)-1) = \sum_{k=1}^{j} d_kf_k> 0\;.
$$
Let $A:=\{ k \in [n]: k< j^\cD_p \}$ and $B:=[n]\sm A$. 
Observe that the $f_k$s are increasing with respect to $k$, and thus $k\in B$ implies $f_k>0$.
Let $\sigma$ be a permutation such that $d_{\sigma_1}'\leq \ldots \leq d_{\sigma_n}'$
where we set $d_k':=0$ if vertex $k$ is deleted.
Note that $d_k'f_k'=0$ if $d_k'=0$.
Thus $R^{\cD'}_p$ does not change if we add isolated vertices to $\cD'$.
For the sake of simplicity, we consider $\cD'$ to be a degree sequence on $[n]$ with isolated vertices.
Let $A':=\{k\in [n]\colon\sigma_k< j^p_{\cD'} \}$ and $B':=[n]\sm A'$.

Let $T:=\{k\in B: f_k'\geq f_k/2\}$.
Observe that if we delete an edge $ij$ from $G$,
then $f_i$ and $f_j$ decrease by $p$, respectively.
Thus 
\begin{align*}
	\sum_{k\in B\sm T}f_k \leq2\sum_{k\in B\sm T}(f_k-f_k') = 2p \sum_{k\in B\sm T}r_k  \leq  4p \nu n.
\end{align*}
Hence
\begin{align}\label{eq:T}
	\sum_{k\in T}f_k =  R^\cD_p- \sum_{k\in B\sm T}f_k \geq \mu n - 4p \nu n 
	\geq \frac{4\mu}{5}n\;.
\end{align} 
Suppose  $k\in T$.
Then,
\begin{align*}
	d_k' = \frac{1+f'_k}{p}+1\geq \frac{1+f_k/2}{p}+1=\frac{1+f_k}{2p}+\frac{1}{2p}+1 =
	 \frac{d_k}{2}+ \frac{1}{2p} +\frac{1}{2}
	\geq \frac{d_k}{2}\;.
\end{align*}
We define $c:=d_{j^\cD_p}$.
It follows that,
\begin{align}\label{eq:T1}
	\sum_{k\in T}d_k'f_k' 
	\geq \frac{1}{4}\sum_{k\in T}d_kf_k
	\geq \frac{c}{4} \sum_{k\in T}f_k
	\stackrel{(\ref{eq:T})}{\geq } \frac{c\mu}{5}n\;.
\end{align} 
Let $T_1\subseteq T$ be such that
\begin{align}\label{eq:T1}
	\sum_{k\in T_1} d_k'f_k'> \frac{c \mu}{20}n\;,
\end{align}
and $\max\{\sigma_k:\, k \in T_1\}$ is minimized
 (choose the set of consecutive vertices in $T$ smallest with respect to the order ${\sigma_1},\ldots,{\sigma_n}$).
By~\eqref{eq:T1} such a set exists.
Let $k_{max}=\arg\max\{\sigma_k:\, k\in T_1\}$. 
So the above definition implies that 
$$ \sum_{k\in T_1\sm\{k_{max}\}} d_k'f_k' \leq \frac{c \mu}{20}n\;.$$
Observe  also that $f_{k_{max}}'\leq p\Delta(G)\leq \mu n/ 40$ and that 
$d'_k\geq d_k/2\geq c/2$ for each $k \in T_1$ (whereby $\frac{2}{c}d_k' \geq 1$).
We conclude that 
\begin{eqnarray}
	\sum_{k\in T\sm T_1} f_k'\notag
	&= &\sum_{k\in T} f_k'-\sum_{k\in T_1\sm\{k_{max}\}} f_k' - f_{k_{max}}'\\ \notag
	&\geq& \frac{1}{2} \sum_{k\in T} f_k- \frac{2}{c}\sum_{k\in T_1\sm\{k_{max}\}} d_k'f_k' - f_{k_{max}}'\\ 
	&\stackrel{\eqref{eq:T}}{\geq} &\frac{2\mu }{5}n - \frac{\mu }{10}n -  \frac{\mu }{40}n 
	\geq \frac{\mu }{4}n\;. \label{eq:ToutT|}
\end{eqnarray}

Let $B_1:=\{k\in B: f_k'<0\}$ and note that $B_1\subseteq B\setminus T$ and $B_1\subseteq A'$.

Recall that $d'_k=d_k-r_k$ and that $f'_k=f_k-pr_k$. By the definition of $A$ and $c$, we observe that 
$\sum_{k\in A} d_k f_k \geq -c(p(c-1)-1)$. Thus
\begin{align*}
	\sum_{k\in A} d_k'f_k'& \geq  \sum_{k\in A} d_kf_k - \sum_{k\in A} r_k(f_k+p d_k)\\
	&\geq -c(p(c-1)-1) -\sum_{k\in A} r_k(p(2c-1)-1) \\
	&\geq  - pc^2 - 4 c p \nu n\;.
\end{align*}
If $k\in B_1$, then $f_k>0$ and from $f'_k<0$, we obtain $d_k< {1}/{p}+r_k+1$.
Thus
\begin{align*}
	\sum_{k\in B_1} d_k'f_k'& =  \sum_{k\in B_1} \left( (d_k - r_k)f_k- pr_k(d_k -r_k) \right) \\
	&\geq  -\sum_{k\in B_1} pr_k(d_k -r_k) \\
	&>  -\sum_{k\in B_1} r_k(1+p)\\
	&\geq - 2(1+p)\nu n
	\geq - 4cp\nu n\;,
\end{align*}
where we used that $1/c\leq p$ in the last line.
Since $A$ and $B_1$ are disjoint, we deduce that
\begin{align*}
	\sum_{k\in A \cup B_1} d_k'f_k'	&\geq  - pc^2 - 8 c \nu pn\;.
\end{align*}
%
%
Let $T_2\subseteq T$ be such that
\begin{align*}
	\sum_{k\in T_2} d_k'f_k'> pc^2 + 8 c \nu pn\;,
\end{align*}
and $\max\{\sigma_k:\, k \in T_2\}$ is minimized (choose the set of consecutive vertices in $T$ smallest with respect to the order ${\sigma_1},\ldots,{\sigma_n}$).
As $\Delta(G)\leq \mu n/(40p)$, we conclude that $pc^2 \leq c \mu n/40$.
Since $8 c \nu pn \leq c \mu n/40$ by hypothesis, using~\eqref{eq:T1} we conclude that $T_2$ exists. Note also that $T_2\subseteq T_1$.
Therefore, by using \eqref{eq:ToutT|}, we obtain
\begin{align}\label{eq:large_T_T_2}
\sum_{k\in T\sm T_2} f_k'\geq \frac{\mu n}{4}\;.
\end{align}
Since $A\cup B_1$ and $T_2$ are disjoint, we conclude that 
\begin{align}\label{eq:no_ideas}
\sum_{k\in A \cup B_1 \cup T_2} d_k'f_k'>0\;.
\end{align}
The previous inequality suggests that the vertices in $T\setminus T_2$ might belong to $B'$ and thus contribute to $R_p^{\cD'}$.
However, in the new ordering $\sigma$, there might be vertices in $A$ with larger degree than some of the vertices in $T\setminus T_2$.
Let $P:=(T \setminus T_2)\cap A'$.
If $\sum_{k\in P}f_k'\leq \frac{\mu n}{8}$, 
then~\eqref{eq:large_T_T_2} implies
$$
R_p^{\cD'}\geq \sum_{k\in (T\setminus T_2)\cap B'}f_k'\geq \frac{\mu n}{4}- \frac{\mu n}{8}= \frac{\mu n}{8}\geq \frac{\mu n'}{8}\;,
$$
and we are done.

Hence, we may assume $\sum_{k\in P}f_k' > \frac{\mu n}{8}$.
Recall that, since $P\subseteq T$,
we have $d_k'\geq d_k/2 \geq c/2$ for every $k\in P$
and hence, by~\eqref{eq:no_ideas},
\begin{align}\label{eq:eqq}
	\sum_{k\in A \cup B_1 \cup T_2 \cup P} d_k'f_k'
	> \frac{c\mu n}{16} \;.
\end{align}
Thus a significant amount of vertices in $A \cup B_1 \cup T_2 \cup P$ need to be in $B'$. Let $Q:=(A \cup B_1 \cup T_2 \cup P) \cap B'$. Since $f'_k<0$ only for $k\in A\cup B_1$, and $f'_k>0$ all $k\in B'$,~\eqref{eq:eqq} implies
\begin{align}\label{eq:eqq2}
	\sum_{k\in Q} d_k'f_k'
	> \frac{c\mu n}{16} \;.
\end{align}
Observe that $B_1\cup P\subseteq A'$.
By our choice of $T_2$, the degree of a vertex in $T_2$ in $G'$ is at most the degree of a vertex in $P$ in $G'$; 
that is, vertices in $T_2$ are smaller than vertices in $P$ with respect to the ordering $\sigma$.
Thus if a vertex of $T_2$ is contained in $Q$, then $P=\emptyset$ and this a contradiction to our assumption.
Therefore, $Q=A \cap B'$ and hence $d_k'\leq d_k\leq c$ for each $k\in Q$.
Using~\eqref{eq:eqq2} we obtain,
\begin{align*}
	R_p^{\cD'} \geq \sum_{k\in B'} f_k' 
\geq \sum_{k\in Q} f_k' 
\geq \frac{1}{c}\sum_{k\in Q} d_k' f_k' 
	\geq \frac{1}{c}\cdot \frac{c\mu }{16}n
	\geq \frac{\mu }{16}n\;.
\end{align*}
This completes the proof.
\end{proof}

\section{Proof of Proposition~\ref{prop:explo}}\label{sec:explo}

Let $G^\cD$ be a graph chosen according to the uniform distribution on $\cG^\cD$. 
The proof of Proposition~\ref{prop:explo} will consist of a careful analysis of an exploration of the different components of $G^\cD_p$ and 
will heavily rely on the switching method. 
Our proofs are similar in spirit to those in~\cite{joos2016how}, but the additional  
level of randomness that is due to bond percolation makes the arguments more involved
and additional arguments are needed.

In order to bound the number of switches it is more convenient to denote by $d(u)$ the degree of a vertex $u\in V$, instead of using $d_i$ for $i\in V$. We will use this notation in this and in the next section.

\subsection{Connection probabilities via the switching method}

The following three technical lemmas 
provide the necessary tools 
needed for proving that the random exploration process follows closely to what we expect it to do.
To prove these lemmas we make extensive use of the switching method and, in particular,  of inequality~(\ref{eq:fund}).

\begin{lemma}\label{lem: vertices adjacent}
Suppose $n\in \N$ and $1/n \ll 1$ and let  $Z'\subseteq V$. 
Suppose $H'$ is a graph with vertex set $Z'$ and
$F'$ is a bipartite graph with vertex partition $(Z',V\sm Z')$. 
Suppose $u\in Z'$ and $v\in V\sm Z'$ such that  $uv\notin E(F')$.
Suppose $\cD$ is a degree sequence on $V$ such that 
\begin{enumerate}[label={\rm (E\arabic*)}]
\item\label{item:E1} $d(u)\leq n^{1/4}$,
\item\label{item:E2} if $w\in V$ and $d(w)> n^{1/4}$, then $w \in Z'$ and $d(w)=d_{H'}(w)$, and
\item\label{item:E3} $\sum_{w\in  V\sm Z'}(d(w)-d_{F'}(w)) \geq n/20$.
\end{enumerate}
Then,
\begin{align*}
	\Pro[uv\in E(G^\cD)\mid G^\cD[Z']=H',\, F'\subseteq G^\cD]\leq 30 n^{-1/2}.
\end{align*}
\end{lemma}
\begin{proof}
Let $\cF^+$ be the set of graphs with degree sequence $\cD$ such that $G[Z']=H'$, $F'\subseteq G$ and $uv\in E(G)$,
and let $\cF^-$ be the set of graphs with degree sequence $\cD$ such that $G[Z']=H'$, $F'\subseteq G$ and $uv\notin E(G)$. 
We will only perform switches that involve edges that are not contained in $E(H')\cup E(F')$. 
This ensures that the graph $G_0$ obtained from a switch also satisfies $G_0[Z']=H'$ and $F'\subseteq G_0$. 
As this is the first proof that involves the switching method, we will provide an extra level of detail. 

For every $G\in \cF^+$, 
let $s^+(G)$ be the number of switches that transform $G$ into a graph in $\cF^-$. 
We seek for a  lower bound on $s^+(G)$. 
Indeed, we will find many edges $xy$ such that the $\{uv,xy\}$-switch leads to a graph in $\cF^-$. 
For this, it suffices to select an edge $xy$ such that $xy$ is at distance at least $2$ from $uv$, we have $xy\notin E(F')$, and $x\in V\sm Z'$. 
By~\ref{item:E3}, there are at least $n/20$ edges that have one endpoint in $V\sm Z'$ and are not contained in $E(F')$. 
Therefore, it suffices to count how many of them lie at distance at most $1$ from $uv$. Note that $d(u),d(v)\leq n^{1/4}$. 
Moreover, $v$ has no neighbour with degree larger than $n^{1/4}$. 
While $u$ can have neighbours $w\in Z'$ with degree larger than $n^{1/4}$, all the edges incident to $w$ have both endpoints in $Z'$ (by \ref{item:E2}). 
It follows, that there are at most $2n^{1/2}$ edges at distance at most $1$ from $uv$ with at least one endpoint in $V\sm Z'$. 
Note that for any such $xy$, 
the $\{uv,xy\}$-switch transforms $G$ into a simple graph $G_0$ with degree sequence $\cD$,  $G_0[Z']=H'$, $F'\subseteq G_0$, and $uv\notin E(G_0)$. Therefore, 
$$
s^+(G)\geq \frac{n}{20}-2n^{1/2} \geq \frac{n}{30}\;.
$$

For every $G\in \cF^-$, let $s^-(G)$ be the number of switches that transform $G$ into a graph in~$\cF^+$. 
We bound $s^-(G)$ from above. 
Clearly, any such switch is of the form $\{ux,vy\}$ for some $x,y\in V$. Since $d(u),d(v)\leq n^{1/4}$, 
there are at most $d(u)d(v)\le n^{1/2}$ choices for the edges $ux$ and $vy$.
Therefore,
$$
s^-(G)\leq n^{1/2}\;.
$$
Using~\eqref{eq:fund} we obtain
\begin{align*}
	\Pro[uv\in E(G^\cD) \mid G^\cD[Z']=H',\, F'\subseteq G^\cD]
	&\leq \frac{s^-(G)}{s^+(G)}\cdot\Pro[uv\notin E(G^\cD) \mid G^\cD[Z']=H',\, F'\subseteq G^\cD]\\
	&\leq \frac{30n^{1/2}}{n}\cdot \Pro[uv\notin E(G^\cD)\mid G^\cD[Z']=H',\, F'\subseteq G^\cD]\\
	&\leq 30 n^{-1/2}\;.
\end{align*}
\end{proof}

For a graph $G$, a vertex $v$ and set of vertices $S$ in $G$, we write $d_G(v,S)$ for the number of neighbours of $v$ in $S$.

\begin{lemma}\label{lem: back edges}
Suppose $n\in \N$ and $1/n\ll \nu\ll 1$. 
Suppose  $Z'\subseteq V$. 
Suppose $H'$ is a graph with vertex set $Z'$, and $F'$ is a bipartite graph with vertex partition $(Z',V\sm Z')$. 
Suppose $x\in V\sm Z'$  and $z\in Z'$.
Suppose $\cD$ is a degree sequence on $V$ such that
\begin{itemize}
\item[-]~if $w\in V$ and $d(w)> n^{1/4}$, then $w \in Z'$ and $d(w)=d_{H'}(w)$,
\item[-] $\sum_{w\in Z'} (d(w)-d_{H'}(w)-d_{F'}(w)) \leq \nu n$, 
\item[-] $\sum_{w\in V\sm Z'}(d(w)-d_{F'}(w)) \geq n/20$, and
\item[-] $xz\notin E(F')$,
\end{itemize}
Then, for every $i\geq 0$ and $Z'':=Z'\sm \{z\}$,
\begin{align*}
	\Pro[d_{G^\cD}(x,Z'')-d_{F'}(x)> \lfloor \sqrt{\nu}(d(x)-d_{F'}(x))\rfloor +i \mid G^\cD[Z']=H',\; F'\subseteq G^\cD,\; \cE]\leq (22\sqrt{\nu})^{i+1}\;,
\end{align*}
where $\cE\in \{\{ xz\in E(G^\cD) \},\{ xz\notin E(G^\cD) \}\}$. 
Therefore, by averaging, we also have
\begin{align*}
	\Pro[d_{G^\cD}(x,Z'')-d_{F'}(x)> \lfloor \sqrt{\nu}(d(x)-d_{F'}(x))\rfloor +i \mid G^\cD[Z']=H',\; F'\subseteq G^\cD]\leq (22\sqrt{\nu})^{i+1}\;.
\end{align*}

\end{lemma}
\begin{proof}
Let $K:=\lfloor\sqrt{\nu}(d(x)-d_{F'}(x))\rfloor$. For every $k\geq K$, let $\cF_k = \cF_k (\cE)$ be the set of graphs $G$ with degree sequence $\cD$ such that $G[Z']=H'$, $F'\subseteq G$, 
$d_G(x,Z'')-d_{F'}(x)=k$, and $\cE$ is satisfied. 
As before, we will only perform switches using edges that are not contained in  $E(H')\cup E(F')$.

Consider a graph in $\cF_k$. 
Then in any of the two possibilities for $\cE$, there are at most $(d(x)-d_{F'}(x))\nu n$ switches that lead to a graph in $\cF_{k+1}$.

For every graph in $\cF_{k+1}$, arguing similar as in Lemma~\ref{lem: vertices adjacent}, there are at least $(k+1)({n}/{20}-\nu n- 2n^{1/2})\geq {(k+1)n}/{21}$ switches that lead to a graph in $\cF_k$. 
This is the number of pairs of edges where one element is among the $k+1$ edges between $x$ and $Z''$ and which are not contained in $E(F')$,
and the other element is among the edges with both endpoints in $V\sm Z'$ (at least ${n}/{20}-\nu n - 2n^{1/2}$) which are at distance at least 2 from 
the endpoints of the first element. 

Thus, for $k\geq K$, we obtain
\begin{align*}
	\Pro[\cF_{k+1}]\leq \frac{21(d(x)-d_{F'}(x))\nu n}{(k+1) n}\Pro[\cF_{k}] \leq 22\sqrt{\nu}\Pro[\cF_{k}]\;,
\end{align*}
which implies that
\begin{align*}
	\Pro[d_{G^\cD}(x,Z'')-d_{F'}(x)> K+i \mid  G^\cD[Z']=H',\; F'\subseteq G^\cD,\cE]\leq (22\sqrt{\nu})^{i+1}\;.
\end{align*}
\end{proof}

\begin{lemma}\label{lem: prob next}
Suppose $n\in \N$ and $1/n \ll \nu \ll 1$.
Suppose $Z\subseteq V$. 
Suppose $H$ is a graph with vertex set $Z$ and $F$ is a bipartite graph with vertex partition $(Z,V\sm Z)$. 
Suppose $z \in Z$ and $x\in V\sm Z$ such that $xz\notin E(F)$.
Suppose $\cD$ is a degree sequence on $V$ (write $\hat{d}(u)=d(u)-d_H(u)-d_F(u)$ for all $u\in V$) such that
\begin{itemize}
\item[-]~if $w\in V$ and $d(w)> n^{1/4}$, then $w \in Z$ and $d(w)=d_H(w)$,
\item[-] $\sum_{w\in Z} \hat{d}(w) \leq \nu n$, and
\item[-] $M:=\sum_{w\in V\sm Z}\hat{d}(w) \geq n/10$.
\end{itemize}
\noindent Then,
\begin{align*}
	\Pro[xz \in E(G^\cD) \mid G^\cD[Z]=H,\, F\subseteq G^\cD  ]= \frac{\hat{d}(x)\hat{d}(z)}{M}(1\pm 25\sqrt{\nu}).
\end{align*}
\end{lemma}
\begin{proof}
Let $\cF_{xz}^+$ be the set of graphs $G$ with degree sequence $\cD$ such that $G[Z]=H$, $F\subseteq G$ and $xz \in E(G)$ and $\cF_{xz}^-$ the set of graphs with degree sequence $\cD$ such that 
$G[Z]=H$, $F\subseteq G$ but $xz \not \in E(G)$. 
As before, we consider only switches using edges that are not contained in $E(H)\cup E(F)$.

First, note that if $\min\{\hat{d}(x),\hat{d}(z)\}=0$, then the statement holds trivially. 
Therefore, we may assume that $\hat{d}(x),\hat{d}(z)\geq 1$. 
Suppose $G\in \cF_{xz}^-$. 
Applying Lemma~\ref{lem: back edges} with $Z'=Z$, $H'=H$, $F'=F$ and $i=0$, we deduce that
\begin{align*}
	\Pro[d_{G}(x,Z)-d_F(x)\geq \sqrt{\nu}\hat{d}(x)\mid G[Z]=H,F\subseteq G,xz\notin E(G)] \leq 22\sqrt{\nu}\;.
\end{align*}
Let $\hat{\cF}_{xz}^-$ denote the subset of $\cF_{xz}^-$ where
$d_G(x,Z)< d_F(x) + \sqrt{\nu} \hat{d}(x)$ holds. 
Then the above 
implies that 
\begin{align}\label{eq:F-}
	|\hat{\cF}^-_{xz}| \geq (1 - 22 \sqrt{\nu}) |\cF_{xz}^-|. 
\end{align}
In other words, for at least $(1- 22\sqrt{\nu})|\cF_{xz}^-|$ of 
the graphs in $\cF_{xz}^-$,
the vertex $x$ has at most $\sqrt{\nu}\hat{d}(x)$ neighbours $z'\in Z\sm\{z\}$ with $xz'\notin E(F)$.

Since $\hat{d}(z)\geq 1$, 
the vertex $z$ has at least one neighbour $V\sm Z$ through an edge not in $E(F)$.
We now partition the set $\hat{\cF}_{xz}^-$ into sets according to the neighbours of 
$z$ in $V \sm Z$ and the neighbours of $x$ in $Z$ (through edges that do not belong to $E(F)$). 
We will use $\bar{y}$ to denote sets of vertices in
$\{y_1,\ldots, y_r\}\subseteq V\sm (Z\cup \{x\})$
and $\bar{z}$ to denote sets of vertices in $\{z_1,\ldots, z_m \} \subseteq Z\sm \{z\}$.  
We define $\hat{\cF}^-_{xz} (\bar{y}, \bar{z})$ to be the subset of graphs in 
$\hat{\cF}^-_{xz}$ such that
the vertices in $\bar{y}$ are the neighbours of $z$ in $V\sm Z$
and the vertices in $\bar{z}$ are the 
neighbours of $x$ in $Z$. In both cases, we only consider the neighbours that are connected to either $z$ or $x$ by an edge not in $E(F)$. 

Thus, $\hat{\cF}_{xz}^-$ is the disjoint union of all subsets 
$\hat{\cF}^-_{xz} (\bar{y}, \bar{z})$, ranging over 
all $\bar{y}$ and $\bar{z}$ as specified above; that is, in particular, $|\bar{y}| = \hat{d}(z)$ and $|\bar{z}|\leq \sqrt{\nu} \hat{d} (x)$. 
We will now use Lemma~\ref{lem: vertices adjacent} to show that for most members of 
$\hat{\cF}^-_{xz} (\bar{y}, \bar{z})$, the vertex $x$ is not adjacent to any vertex in $\bar{y}$.  

To apply Lemma~\ref{lem: vertices adjacent}, we set $Z' := Z \cup \{x\}$, 
$V(H')=Z'$, and 
$E(H')$ consists of $E(H)$, 
the edges that join $x$ and $\bar{z}$, and the edges in $F$ that are incident to $x$. 
The graph $F'$ is the bipartite graph with vertex set $(Z', V\sm Z')$ and edge set $E(F)\sm \{x\hat{z}: \hat{z}\in Z\}$. 
Also observe that \ref{item:E1}, \ref{item:E2}, and \ref{item:E3} are satisfied; 
in particular, $\sum_{w\in V\sm Z'} (d(w)-d_{F'}(w))\geq  M - d(x)\geq n/20$ holds.
Let $\hat{\cF}^{--}_{xz} (\bar{y}, \bar{z})$ be the subset of $\hat{\cF}^-_{xz} (\bar{y}, \bar{z})$ in which 
$x$ is not adjacent to a vertex in $\bar{y}$. 
Since $xy\notin E(F')$ for each $y\in \bar{y}$, Lemma~\ref{lem: vertices adjacent} implies that 
\begin{align}\label{eq:F--}
	|\hat{\cF}^{--}_{xz} (\bar{y},\bar{z}) | \geq 
	(1 - 30 n^{-1/2} n^{1/4}) |\hat{\cF}^{-}_{xz} (\bar{y}, \bar{z}) |
	= (1 - 30 n^{-1/4}) |\hat{\cF}^{-}_{xz} (\bar{y}, \bar{z}) |,
\end{align}
because $\bar{y}$ contains at most $\hat{d}(z)\leq n^{1/4}$ vertices.

Next, we partition the set $\hat{\cF}^{--}_{xz} (\bar{y}, \bar{z}) $ according to the neighbours of $x$ in $V\sm Z$. 
We will use $\bar{w}$ to denote the set of neighbours of $x$ in $V\sm Z$. 
Thus $\bar{w}$ does not contain any member of $\bar{y}\cup\{x\}$ and 
$(1- \sqrt{\nu}) \hat{d}(x)\leq |\bar{w}|\leq \hat{d}(x)$. 
For such a $\bar{w}$, we let 
$\hat{\cF}^{--}_{xz} (\bar{y}, \bar{z},\bar{w})$ be the subset of 
$\hat{\cF}^{--}_{xz} (\bar{y}, \bar{z})$ 
where $\bar{w}$ are the neighbours of $x$ in $V\sm Z$.

Assume now that $\bar{y} = \{y_1,\ldots, y_{r}\}$
and $\bar{w} = \{w_1,\ldots, w_\ell \}$, with $r=\hat{d}(z)$ and
$(1- \sqrt{\nu}) \hat{d} (x) \leq \ell \leq \hat{d}(x)$.
We fix some $i\in [r]$ and $j\in [\ell]$.
An straightforward switching argument as for example performed in Lemma~\ref{lem: vertices adjacent}
shows that for at least $(1-n^{-1/10})|\hat{\cF}^{--}_{xz} (\bar{y}, \bar{z},\bar{w})|$ graphs in $\hat{\cF}^{--}_{xz} (\bar{y}, \bar{z},\bar{w})$,
the edge $y_iw_j$ is not present.
In this case, we apply the switch $\{zy_i,xw_j\}$.
Thus, in total, the number of switches from graphs in $\hat{\cF}^{--}_{xz} (\bar{y}, \bar{z},\bar{w})$ to graphs in $\cF_{xz}^+$ is at least
\begin{align*}
	(1-n^{-1/10})\ell r |\hat{\cF}^{--}_{xz} (\bar{y}, \bar{z},\bar{w})|
	\geq (1-n^{-1/10})(1- \sqrt{\nu}) \hat{d}(x)\hat{d}(z)|\hat{\cF}^{--}_{xz} (\bar{y}, \bar{z},\bar{w})|.
\end{align*}
Hence the number of switches from graphs in $\hat{\cF}^{-}_{xz} (\bar{y}, \bar{z})$ to graphs in $\cF_{xz}^+$ is at least
\begin{align*}
	(1-n^{-1/10})(1- \sqrt{\nu}) \hat{d}(x)\hat{d}(z) |\hat{\cF}^{--}_{xz} (\bar{y}, \bar{z})|
	\stackrel{(\ref{eq:F--})}{\geq} (1- 2\sqrt{\nu}) \hat{d}(x)\hat{d}(z)|\hat{\cF}^{-}_{xz} (\bar{y}, \bar{z})|.
\end{align*}
This in turn implies that the number of switches from graphs in $\cF_{xz}^-$ to graphs in $\cF_{xz}^+$
is at least
\begin{align*}
	(1- 2\sqrt{\nu}) \hat{d}(x)\hat{d}(z)|\hat{\cF}^{-}_{xz}|
	\stackrel{(\ref{eq:F-})}{\geq} (1- 24\sqrt{\nu}) \hat{d}(x)\hat{d}(z)|\cF^-_{xz} |.
\end{align*}

Furthermore, since the edges of $F$ are not involved in such switches, 
there are at most $\hat{d}(x)\hat{d}(z)$ switches transforming a graph in $\cF_{xz}^-$ into a graph in $\cF_{xz}^+$.

Consider now a graph in $\cF_{xz}^+$. Any switch that transforms it into a graph in $\cF_{xz}^-$ must use the edge $xz$. It suffices to bound the number of choices for the other edge.
On the one hand, it is easy to see that there are at most $M$ switches leading to a graph in $\cF_{xz}^-$.
On the other hand, since $d(x),d(z)\leq n^{1/4}$, 
there are at least $M- \nu n - 2n^{1/2}$ edges in distance at least $2$ from $xz$ which belong to $G[V\sm Z]$.
Thus there are at least $M- 2\nu n$ switches leading to a graph in $\cF_{xz}^-$.
\smallskip

Combining all four bounds, leads to the desired statement.
\end{proof}

\subsection{The exploration process}

In order to bound the order of the largest component in $G^\cD_p$ we will perform an exploration process on $G^\cD$ that reveals 
the components of $G^\cD_p$.  
An \emph{input} is a pair $(G,\fS)$ with the following properties.
For a given degree sequence $\cD$ on the vertex set $V$,
we let $G$ be a graph on $V$ with degree sequence $\cD$
and for every vertex $v$, 
we arbitrarily assign the labels $1,\ldots,d(v)$ to its incident edges.
In this way, each edge obtains two labels. Since each label is associated with one of the endpoints of the corresponding edge, it is convenient to understand this labelling as a labelling of the \emph{semi-edges} of the graph in such a way that the semi-edges incident to $v$ are given the labels $1,\ldots,d(v)$. Thus, during the exploration process, $G$ is equipped with an arbitrary labelling of the semi-edges incident to each vertex. The semi-edge labelling fits well with the switching method: if $G'$ is obtained from $G$ by switching two edges, then the semi-edges of $G'$ naturally inherit the labelling on the semi-edges of $G$. 
The set $\fS=\{\sigma_v: v\in V\}$ is a collection of permutations, one for each vertex $v\in V$, where $\sigma_v$ is a permutation of length $d(v)$.
For technical reasons that will become apparent soon, we will need to consider the exploration process on an input. The labelling on the semi-edges together with $\fS$, will determine the order in which the vertices are explored during the process.


Given an input $(G,\fS)$ and a subset of vertices $S_0\subseteq V$, we proceed to describe the exploration of $G$ from $S_0$. 
First, for every vertex in $v\in V$, 
we permute the labels of its incident semi-edges according to $\sigma_v$.
Observe that a uniformly selected set of permutations $\fS$
leads to a uniformly selected labelling of the semi-edges incident to each vertex of $G$.
First, we expose the graph $G[S_0]$. 
For every $t\geq 0$, let $S_t$ be the set of vertices that have been explored up to time $t$, let $H_t:=G[S_t]$ and let $F_t$ be the bipartite subgraph with vertex partition $(S_t,V\sm S_t)$ that contains those edges of $E(G)$ that have been exposed but have not survived the random deletion -- we will be referring to these edges
as the edges that have \emph{failed to percolate}. 
For a vertex $u\in V$, we define its \emph{free degree at time $t$} as
$$
\hat{d}_t(u):= d(u) - d_{H_t}(u) - d_{F_t}(u)\;.
$$

We may assume that $V$ has some fixed ordering.
If at time $t$ there exists at least one vertex $v\in S_t$ with $\hat{d}_t(v)\geq 1$, 
we select\footnote{To be precise, the selection of a new vertex and the updates of the considered parameters happen between time $t$ and $t+1$.} 
the smallest vertex $v_{t+1}\in S_t$ such that $\hat{d}_t(v_{t+1})\geq 1$.
Let $w_{t+1}$ be the vertex $w\in V \sm {S_t}$ with $v_{t+1}w\in E(G)\sm E(F_t)$ that
minimizes $\sigma_{v_{t+1}}(\ell(w))$, where $\ell(w)$ is the label of the semi-edge incident to $v_{t+1}$ that 
corresponds to $v_{t+1}w$.
After that, with probability $p$, we retain the edge $v_{t+1}w_{t+1}$ in $G_p$. If the edge survives percolation, we proceed as follows: 
\begin{enumerate}[label=\arabic*.]
\item we set $S_{t+1}:=S_t\cup\{w_{t+1}\}$;
\item we expose all the edges (\emph{back edges}) from $w_{t+1}$ to $S_t\sm\{v_{t+1}\}$ that are not in $F_t$; we define the \emph{backward degree}\footnote{Note that the backward degree does not include the contribution of $v_{t+1}$.} of $w_{t+1}$ as 
$$
 d_t'(w_{t+1}) := d(w_{t+1}, S_t\sm\{v_{t+1}\} )-d_{F_t}(w_{t+1}) \;;
$$
\item we retain each of the back edges in $G_p$ independently with probability $p$; and 
\item we define $H_{t+1}:=G[S_{t+1}]$ and 
let $F_{t+1}$ be the bipartite subgraph with vertex partition $(S_{t+1}, V\sm S_{t+1})$ that contains all the edges between $S_{t+1}$ and $V\sm S_{t+1}$ that have failed to percolate so far.
\end{enumerate}
If $v_{t+1}w_{t+1}$ fails to percolate, we set $S_{t+1}:=S_t$, $H_{t+1}:=H_t$, $V(F_{t+1}):=V(F_{t})\cup \{w_{t+1}\}$, and
$E(F_{t+1}):=E(F_t)\cup \{v_{t+1}w_{t+1}\}$.

Finally, if there is no $v\in S_t$ with $\hat{d}_t(v)\geq 1$, we let $w_{t+1}=u$, where $u\in V\sm S_t$ is chosen with probability proportional to $\hat{d}_t(u)$, and 
we set $S_{t+1}:=S_t\cup \{w_{t+1}\}$, $H_{t+1}:=G[S_{t+1}]$ and $F_{t+1}:=F_t$. This marks the beginning of a new component. 
\medskip

Note that at time $t$ we have explored at most $t$ new vertices and that $G[S_{t}]$ is fully exposed (as well as $G_p[S_{t}]$). 
Moreover, there is a 
set of edges $E(F_{t})$ joining $S_{t}$ and $V\sm S_{t}$ that have also been exposed but failed to percolate. 


Let $\cH_t$ denote the history of the exploration process after $t$ rounds (at time $t$). 
More precisely, this is the random object composed of the collection of all the choices that have been made in the exploration process up to time $t$, and include the choice of $S_t$, $H_t=G[S_t]$ and $F_t$. 
Observe that for a fixed input $(G,\fS)$,
the only randomness in this exploration process stems from the percolation process and the random selection of a new vertex if $H_t$ is a union of components in the percolated graph.

The next two variables will be crucial to control our exploration process at time $t$:
\begin{itemize}
\item[-] $M_t:= \sum_{u\in V\sm S_t}\hat{d}_t(u)$, which equals the number of ordered edges $uv$ with $u\in V\sm S_t$ and $uv\notin E(F_t)$.

\item[-] $X_t:=\sum_{u\in S_t} \hat{d}_t(u)$, which equals the number of edges $uv$ with $u\in S_t$, $v\in V\sm S_t$ and $uv\notin E(F_t)$. 
\end{itemize}

The variable $X_t$ counts the number of edges that are suitable to be used in the step $t+1$ to continue the exploration process. If $X_t=0$, then we have completed the exploration of a component of $G_p$.

In order to deduce Proposition~\ref{prop:explo}, 
we will analyse the exploration procedure on the input $(G^\cD,\fS)$, where each permutation in $\fS$ is chosen uniformly at random among all permutations of the appropriate length. 
In order to show that the largest component in $G^\cD_p$ is large or small, 
we will consider the evolution of the random process $\{X_t\}_{t\geq 0}$ conditional on its history $\cH_t$, that is, 
the set of all decisions taken up to step $t$. More formally, $\cH_t$ is the $\sigma$-algebra generated by 
all random decision taken up to step $t$. 
Note that now $\cH_t$ does not only depend on the indicator random variables associated with whether the 
edges survive percolation, but also on the random graph $G^\cD$.
Using the method of the deferred decisions, we can generate each random permutation while we perform the exploration process. 
This ensures that, at step $t$, any choice of $w_{t+1}$ satisfying the desired properties is equally possible (see Section 2.2 in~\cite{joos2016how} for a more details).

\subsection{The expected increase of $X_t$}

For every $uv\in E(G)$, let $I(uv)$ be the indicator random variable that the edge $uv$ percolates. 
\medskip

\noindent 
If $X_t>0$,
then the increase of $X_t$ can be written as
\begin{align}\label{eq:change}
X_{t+1}-X_t&= -(1-I(v_{t+1}w_{t+1}))+I(v_{t+1}w_{t+1})((\hat{d}_t(w_{t+1})-2) - 2d_t'(w_{t+1})) \;,
\end{align}
and if $X_t=0$, as
\begin{align}\label{eq:change2}
X_{t+1}-X_t&= \hat{d}_t(w_{t+1})\;.
\end{align}

The next three lemmas use Lemmas~\ref{lem: back edges} and~\ref{lem: prob next}, \eqref{eq:change}, \eqref{eq:change2}, and $\Exp(I(v_{t+1}w_{t+1}))=p$ to provide bounds on $\Exp[X_{t+1}-X_t \mid \cH_t]$ assuming that $t$ is small, 
$M_t$ is large,
and 
$X_t$ is small and for the first lemma also positive.	
\begin{lemma}\label{lem: exp change upper}
Suppose $n\in \N$ and $1/n\ll \beta,\rho, \eta\ll \lambda \ll\mu, 1/\avd, p\leq 1$.
Let $S_0\subseteq V$ and let $\cD$ be a degree sequence on $V$ such that 
$\Td\leq\avd n$
and $\sum_{u\in V\setminus S_0}d (u)(p(d (u)-1)-1)\leq -\mu n.$

Consider the exploration process described above on $(G^\cD,\fS)$ with initial set $S_0$ and suppose $t\leq \rho n$. 
Conditional on $\cH_t$
satisfying $d_{H_t}(w)=d(w)$ for every $w\in V$ with $d(w)> n^{1/4}$, 
$M_t\geq (1- \eta) \Td$,
 and 
$0<X_t\leq \beta n$,
we have
\begin{align*}
	\Exp[X_{t+1}-X_t \mid \cH_t]
	\leq - \lambda .
\end{align*}
\end{lemma}
\begin{proof}
At time $t$, 
there are at most $t$ vertices $u\in V\sm S_t$ such that $\hat{d}_t(u) = 0$. 
This is the case since $d(u)=\hat{d}_t(u)+d_{F_t}(u)$ for all $u\in V\sm S_t$
and at each
step $s\leq t$ there is at most one edge added to $F_s$.
Observe also that the function $h(x)= x(x-2)$ is monotone increasing for $x\geq 1$
and $h(0)=h(2)= 0$, $h(1)=-1$.
This implies that $\hat{d}_t(u)(\hat{d}_t(u)-2)> d (u)(d (u)-2)$ only if $d(u)=1$ and $\hat{d}_t(u)=0$.
It follows that
\begin{align} \label{eq:S_t-S_0}
	\sum_{u\in V\setminus S_t}\hat{d}_t(u)(\hat{d}_t(u)-2) \leq t+ \sum_{u\in V\sm S_0}d (u)(d (u)-2)\; .
\end{align}


The fact that $S_0$ contains all the neighbours in $G^\cD$ of vertices of degree larger than $n^{1/4}$ and that $S_0\subseteq S_t$, 
ensures that for every $v\in S_t$ such that $\hat{d}_t(v)\geq 1$,  we have $d(v)\leq n^{1/4}$. 
Choose $u\in V\sm S_t$.
Since $M_t \geq (1- \eta) \Td \geq n/10$ (recall that degrees are positive) and $X_t\leq \beta n$, and provided $v_{t+1}u\notin E(F_t)$, 
we can apply Lemma~\ref{lem: prob next} with $\nu=\beta$, $Z=S_t$, $H=H_t$, $F=F_t$, $z=v_{t+1}$, and $x=u$ to conclude that
\begin{align*}
	\Pro[v_{t+1}u\in E(G^\cD) \mid \cH_t]= \frac{\hat{d}_t(v_{t+1})\hat{d}_t(u)}{M_t}(1\pm 25\sqrt{\beta})\;.
\end{align*}
Observe that every edge incident to $v_{t+1}$ that is not contained in $E(F_t)\cup E(H_t)$ is chosen with the same probability to continue the exploration process. 
Thus the probability that $u$ is the vertex $w$ that minimizes $\sigma_{v_{t+1}}(\ell(w))$, 
where $\ell(w)$ is the label of the semi-edge incident to $v_{t+1}$ and corresponding to $v_{t+1}w$, 
among all $w\in V\sm S_t$ with $v_{t+1}w\in E(G^\cD)\sm E(F_t)$, is precisely $1/\hat{d}_t(v_{t+1})$.
Therefore,
\begin{align*}
	\Pro[u=w_{t+1}\mid \cH_t]= \frac{\hat{d}_t(u)}{M_t}(1\pm 25\sqrt{\beta})\;.
\end{align*}
Note that if $v_{t+1}u\in E(F_t)$, then $\Pro[u=w_{t+1}\mid \cH_t]=0$.

Let $n_1$ denote the number of vertices $v \in V \sm S_t$ with $\hat{d}_t(v) =1$ and let $A_t\subseteq V\sm S_t$ denote the set of vertices $u$ such that $v_{t+1}u\in E(F_t)$. Since $d(v_{t+1})\leq n^{1/4}$, we have $|A_t|\leq n^{1/4}$. 
Also $\hat{d}_t  (u) < d(u)\leq  n^{1/4}$ for all $u \in A_t$. 
Therefore,  $|\sum_{u\in A_t} \hat{d}_t(u)(\hat{d}_t(u) -2) |\leq n^{3/4}$.

Using~(\ref{eq:change}) and the fact that an edge percolates independently from the underlying graph, we conclude that
\begin{eqnarray*}
	&&\Exp[X_{t+1}-X_t \mid \cH_t] 
	\leq -(1-p)+ p \sum_{u\in V\sm S_t}\Pro[u=w_{t+1}](\hat{d}_t(u)-2)\\
	&\leq& -(1-p)+ \frac{p}{M_t}\left( (1+25\sqrt{\beta})\sum_{u\in V \sm S_t:\hat{d}_t(u)\geq 2}\hat{d}_t(u)(\hat{d}_t(u)-2)-n_1(1- 25 \sqrt{\beta}) + 2n^{3/4}\right)\\
	&\leq& -(1-p)+(1+25\sqrt{\beta})\frac{p}{M_t}\left(\sum_{u\in V \sm S_t}\hat{d}_t(u)(\hat{d}_t(u)-2)\right) + 100\sqrt{\beta}+\frac{2pn^{3/4}}{M_t}\\
	&\stackrel{\eqref{eq:S_t-S_0}, \beta, \eta \ll 1}{\leq}& -(1-p)+(1+25\sqrt{\beta})\frac{p}{M_t}\left(\sum_{u\in V \setminus S_0}d (u)(d (u)-2)\right)+ 2\rho+ 101\sqrt{\beta}.
\end{eqnarray*}
Now, we write	
\begin{align}	 \label{eq:rewrite}
& \frac{p}{M_t}\left(\sum_{u\in V \sm S_0}d(u)(d(u)-2)\right) =
\frac{p}{M_t}\left(\sum_{u\in V \sm S_0}d(u)(d(u)-1) - \sum_{u\in V \sm S_0}d(u)\right) \nonumber \\
&=\frac{p}{M_t}\left(\sum_{u\in V \sm S_0}d(u)(d(u)-1) - \sum_{u\in V \sm S_0}d(u)\right) 
- \frac{1}{M_t}\sum_{u\in V \sm S_0}d(u) +\frac{1}{M_t}\sum_{u\in V \sm S_0}d(u)
 \nonumber\\
	&= \frac{1-p}{M_t} \sum_{u\in V\setminus S_0}d (u) 
	+\frac{1}{M_t}\left(\sum_{u\in V\setminus S_0}d (u)(p(d (u)-1)-1)\right).
\end{align}
Hence, 
\begin{align*}	
\Exp[X_{t+1}-X_t \mid \cH_t] &\leq \frac{1-p}{M_t} \left( -M_t + (1+25\sqrt{\beta})\sum_{u\in V\setminus S_0}d (u)  \right) \\ 
    &\hspace{1cm}+\frac{1+25\sqrt{\beta}}{M_t}\left(\sum_{u\in V\setminus S_0}d (u)(p(d (u)-1)-1)\right)+ 2\rho+ 101\sqrt{\beta}. 
\end{align*}
But since $M_t \geq (1 - \eta)\Sigma^\cD\geq  (1 - \eta)\sum_{v\in V\sm S_0}d (v)$, we have 
\begin{align*}
\frac{1-p}{M_t} \left( -M_t +(1+25\sqrt{\beta})\sum_{u\in V \setminus S_0} d (u)  \right) 
& \leq \frac{1-p}{1-\eta}\left(-(1 - \eta) +1 + 25 \sqrt{\beta}\right) 
\leq \eta +\beta^{1/3},
\end{align*}
where the previous inequality follows as $\eta \leq p$ and $\beta \ll 1$.

Thereby, taking $\beta, \rho, \eta \ll \lambda\ll \mu, 1/\avd$,
\begin{align*}
\Exp[X_{t+1}-X_t \mid \cH_t] 	&\leq \frac{1+25\sqrt{\beta}}{M_t}\left(\sum_{u\in V \setminus S_0}d(u)(p(d (u)-1)-1)\right) +2\rho+ 101\sqrt{\beta}+\eta +\beta^{1/3} \\
&\leq -\frac{\mu}{\avd} + 2\rho+ 101\sqrt{\beta}+\eta +\beta^{1/3} \\
&\leq - \lambda \;,
\end{align*}
which completes the proof.
\end{proof}
%
%

For the following two lemmas we do not require the condition $X_t>0$.
\begin{lemma}\label{lem: exp change lower}
Suppose $n\in \N$ and $1/n\ll \beta,\rho, \eta\ll \lambda \ll\mu, 1/\avd, p\leq1$.
Let $S_0\subseteq V$ and let $\cD$ be a degree sequence on $V$ such that $\Td\leq\avd n$
and $R_p^{\cD} \geq \mu n$.

Consider the exploration process described above  on $(G^\cD,\fS)$ with initial set $S_0$ and suppose $t\leq \rho n$. 
Conditional on $\cH_t$ satisfying $d_{H_t}(w)=d(w)$ for every $w\in V$ with $d(w)> n^{1/4}$, $M_t\geq (1- \eta) \Sigma^\cD$ and $X_t\leq \beta n$,
we have
%
\begin{align*}
\Exp[\hat{d}_t(w_{t+1})-2\mid \cH_t]\geq \frac{2\lambda+1-p}{p}\;.
\end{align*}
\end{lemma}
\begin{proof}
We will first provide a lower bound on $\sum_{u\in V\sm S_{t}}\hat{d}_t(u)(p(\hat{d}_t(u)-1)-1)$. 
Consider a realisation of the degree sequence of $G^\cD[V\sm S_t]$ which satisfies the conditions on $\cH_t$, which we denote by $\cD_t'=(d_1',\dots, d'_{n'})$ with $d_1'\leq \dots\leq d_{n'}$ and $n'=|\cD'_t|$. 
Since $M_t \geq (1- \eta) \Sigma^\cD$, we have that $\sum_{v\in S_t} d(v) \leq \eta \avd n$.
By Proposition~\ref{prop:imp}, with $S=S_t$ and $\nu=\eta \avd$, and since  
$R^\cD_p \geq \mu n$ (observe that we choose $\nu\ll \mu$), 
we have  $ R_p^{\cD_t '}\geq \frac{\mu n}{50}$. 
(At this point we want to stress that the previous bound is not a with-high-probability statement; it holds for every possible realisation of $\cD'_t$.)
Recall that for $j\geq j_{\cD_t'}^p$, we have $p(d'_j-1)-1 > 0$.
It follows that
\begin{eqnarray}\label{eq:RDp}
&&\sum_{u\in V \sm S_{t}} \hat{d}_t(u)(p(\hat{d}_t(u)-1)-1)\notag \\
&\geq& \sum_{u \in V\sm S_t}d_{G^\cD[V\sm S_t]}(u)(p(d_{G^\cD[V\sm S_t]}(u)-1)-1) -X_t\nonumber \\
&\geq & 
\sum_{j=1}^{ j_{\cD_t'}^p}d_j'(p(d'_j-1)-1) +\sum_{j = j_{\cD_t'}^p}^{n'} d_j'(p(d'_j-1)-1) 
- d_{j_{\cD_t'}^p}(p(d_{j_{\cD_t'}^p} -1)-1) - \beta n \nonumber \\
&\geq& \sum_{j = j_{\cD_t'}^p}^{n'} d_j'(p(d'_j-1)-1)- 2\beta n\nonumber\\
&\geq & \sum_{j = j_{\cD_t'}^p}^n(p(d'_j-1)-1) - 2\beta n\nonumber\\
&=& R_p^{\cD_t '} -2\beta n \stackrel{\beta \ll \mu}{\geq} \frac{\mu n}{60}.
\end{eqnarray}

Let $n_1$ denote the number of vertices $v \in V \sm S_t$ with $\hat{d}_t(v) =1$ and let $A_t\subseteq V\sm S_t$ denote the set of vertices $u$ such that $v_{t+1}u\in E(F_t)$. As in Lemma~\ref{lem: exp change upper}, we have $|\sum_{u\in A_t} \hat{d}_t(u)(\hat{d}_t(u) -2) |\leq n^{3/4}$.  

If $X_t>0$ holds\footnote{Observe that this calculation is also correct if $X_t=0$.},
we can use Lemma~\ref{lem: prob next} to show that
\begin{align*} 
	p\Exp[\hat{d}_t(w_{t+1})&-2\mid \cH_t]
	= p \sum_{u\in V\sm S_{t}}\Pro[w_{t+1}=u\mid \cH_t](\hat{d}_t(u)-2)\\
	&\geq \frac{p}{M_t}\left( (1-25\sqrt{\beta})\sum_{u\in V \sm S_t:\hat{d}(u)\geq 2}\hat{d}_t(u)(\hat{d}_t(u)-2)-n_1(1+ 25 \sqrt{\beta}) - n^{3/4}\right)\\
	&\geq (1-25\sqrt{\beta})\frac{p}{M_t}\left(\sum_{u\in V \sm S_t}\hat{d}_t(u)(\hat{d}_t(u)-2)\right) - 101\sqrt{\beta}\;.
\end{align*}
Therefore, using a similar calculation as in (\ref{eq:rewrite}), we conclude
\begin{align}
	-(1-p)+p\Exp[\hat{d}_t(w_{t+1})-2\mid \cH_t]
	&\geq \notag
	\frac{1-p}{M_t} \left( -M_t + (1-25\sqrt{\beta})\sum_{u\in V\setminus S_t}\hat{d}_t (u)  \right) \\
    &\hspace{-2cm}+\frac{1-25\sqrt{\beta}}{M_t}\left(\sum_{u\in V\setminus S_t}\hat{d}_t (u)(p(\hat{d}_t (u)-1)-1)\right)- 101\sqrt{\beta}. \label{eq:part1}
\end{align}
Note that $\sum_{u\in V \setminus S_t} \hat{d}_t (u) \geq M_t - t \geq (1-\rho)M_t$. Similarly as before, 
\begin{align}\notag
\frac{1-p}{M_t} \left( -M_t +(1-25\sqrt{\beta})\sum_{u\in V \setminus S_t} \hat{d}_t (u)  \right) 
& \geq (1-p)\left(-1 +(1 - 25 \sqrt{\beta})(1-\rho)\right) \\
&\geq -(\rho +25\sqrt{\beta}). \label{eq:part2}
\end{align}
Using~\eqref{eq:RDp}, \eqref{eq:part1}, \eqref{eq:part2}, and taking $\beta,\rho\ll \lambda\ll\mu, 1/\avd$, we conclude the proof of the lemma as
\begin{align*}
-(1-p)+&p\Exp[\hat{d}_t(w_{t+1})-2\mid \cH_t]\geq \frac{(1-25\sqrt{\beta})\mu}{60 \avd} -\rho -126\sqrt{\beta} \geq 2\lambda\;.
\end{align*}
\end{proof}

\begin{lemma}\label{lem: back edges2}
Suppose $n\in \N$ and $1/n\ll \beta,\rho, \eta\ll\lambda \ll\mu, 1/\avd, p\leq 1$. 
Suppose $S_0\subseteq V$ and let $\cD$ be a degree sequence on $V$ such that $\Td\leq\avd n$
and $R_p^{\cD} \geq \mu n$.

Consider the exploration process described above  on $(G^\cD,\fS)$ with initial set $S_0$ and suppose $t\leq \rho n$. 
Conditional on $\cH_t$ satisfying $d_{H_t}(w)=d(w)$ for every $w\in V$ with $d(w)> n^{1/4}$, $M_t\geq (1- \eta) \Td$ and $X_t\leq \beta n$,
we have
\begin{align*}
	\Exp[d'_t(w_{t+1})\mid \cH_t]
	\leq \frac{1}{10}\Exp[\hat{d}_t(w_{t+1})-2\mid \cH_t],
\end{align*}
and
\begin{align*}
	\Exp[X_{t+1}-X_t\mid \cH_t]\geq  \lambda\;. 
\end{align*}
\end{lemma}
\begin{proof}
Suppose first that $X_t=0$. Recall that in this case, we start the exploration of a new component and 
we select a vertex in $V\sm S_t$ with probability proportional to its free degree. As this is at least 1, 
we deduce $\Exp[X_{t+1}-X_t\mid \cH_t]\geq p\geq  \lambda$.
By Lemma~\ref{lem: exp change lower}, we have $\Exp[\hat{d}_t(w_{t+1})-2\mid \cH_t]>0$.
As $d'_t(w_{t+1})=0$, the first bound also follows.

Suppose now that $X_t>0$.
In order to bound the expectation of $d'(w_{t+1})$ it is clear from Lemma~\ref{lem: back edges} with $\nu=\beta$, $Z=S_t$, $H=H_t$, $F'=F_t$, $x=w_{t+1}$, $z=v_{t+1}$ and $\cE=\{xz\in E(G^\cD)\}$ that
\begin{align}\label{eq:back_edges}
	\Exp[d_t'(w_{t+1})\mid \cH_t]&
	\leq 2\sqrt{\beta}\Exp[\hat{d}_t(w_{t+1})\mid \cH_t]\nonumber\\
	&=2\sqrt{\beta}\Exp[\hat{d}_t(w_{t+1})-2\mid \cH_t]+4\sqrt{\beta}\\
	&\leq \frac{1}{10}\Exp[\hat{d}_t(w_{t+1})-2\mid \cH_t]\nonumber\;,
\end{align}
where the previous inequality follows from the fact that $\Exp[\hat{d}(w_{t+1})-2\mid \cH_t]\geq 2\lambda$ by Lemma~\ref{lem: exp change lower}
and $\beta\ll \lambda$.

Using~\eqref{eq:change},~\eqref{eq:back_edges}, Lemma~\ref{lem: exp change lower} and taking 
$\beta\ll\lambda$, we obtain
\begin{align*}
	\Exp[X_{t+1}-X_t\mid \cH_t]
	&= -(1-p)+p(\Exp[\hat{d}_t(w_{t+1})-2\mid \cH_t] - 2\Exp[d_t'(w_{t+1})\mid \cH_t])\\
	&\geq -(1-p)+p\cdot (1-4\sqrt{\beta})\Exp[\hat{d}_t(w_{t+1})-2\mid \cH_t] - 8p\sqrt{\beta}\\
	&\geq 2(1-4\sqrt{\beta})\lambda -(1-p)4\sqrt{\beta}- 8p\sqrt{\beta}
	\geq \lambda \;,
\end{align*}
which completes the proof.
\end{proof}

\subsection{Another concentration inequality}

The following lemma will be used to show that several parameters of our process 
do not deviate much from their expected value.

\begin{lemma}\label{lem:our_concentration} 
Suppose $a<0,b>0$, $m\in \{0,1\}$, $t\in \N$, and $y\in [a,0)$.
Suppose $\tau$ is a stopping time with respect to a filtration $(\cF_s)_{s=0}^t$. 
Suppose $Y_0,Y_1,\dots, Y_t$ are random variables such that $Y_s$ is measurable at time $s$ and $Y_s-Y_{s-1}\in [a,b]$.
Suppose that for any $s\in [t]$, we have 
$$ \Exp [\ind{s\leq \tau} (-1)^m (Y_s - Y_{s-1}) \mid \cF_{s-1}] \leq y \ind{s\leq \tau}.  $$ 
Then
\begin{equation*} 
\Pro \left[ (-1)^m (Y_{\tau \wedge t} - Y_0) + \ind{t>\tau} (t-\tau) y > \frac{y t}{2} \right] < e^{-\frac{y^2}{12(b-a)^2}\cdot t}.
\end{equation*}
\end{lemma}
To this end, we shall use the following lemma which was proved in~\cite{ar:SudVond} and is a corollary of a martingale concentration theorem (Theorem~3.12) from~\cite{McD}. 
\begin{proposition}\label{lem:concen}
Let $W_1, \ldots, W_t$ be random variables taking values in $[0,1]$ such that 
$$\Exp[W_s \mid W_1,\ldots,W_{s-1}]\leq w_s$$
for each $s\in [t]$.
Let $\lambda_t :=\sum_{s=1}^t w_s$.
Then for any $0<\delta\leq 1$, we have
$$\Pro\left[\sum_{s=1}^t W_s \geq (1+\delta)\lambda_t \right]\leq e^{- \frac{\delta^2\lambda_t}{3}}.$$
\end{proposition}
\begin{proof}[Proof of Lemma~\ref{lem:our_concentration}] 
The assumption of the lemma implies that for every $s\in [t]$
$$ \Exp [\ind{s\leq \tau} (-1)^m (Y_s - Y_{s-1})  + y \ind{s>\tau} \mid \cF_{s-1}] \leq y.$$
We set  $Z_s := \ind{s\leq \tau}(-1)^m (Y_s - Y_{s-1})  + y \ind{s>\tau}$. The history  $\cF_{s-1}$ completely determines $Z_1,\dots ,Z_{s-1}$, whereby 
\begin{equation} \label{eq:cond_exp} 
\Exp [Z_s \mid Z_1,\dots, Z_{s-1}] \leq y.
\end{equation}

We now rescale the variables $Z_s$ to obtain random variables in $[0,1]$. To this end, we set 
\begin{equation} \label{eq:rescale} 
W_s := \frac{Z_s-a}{b-a}.
\end{equation}
It follows directly from~\eqref{eq:cond_exp} that
$$\Exp [W_s \mid W_1,\dots, W_{s-1}] \leq \frac{y-a}{b-a}=:w\;. $$
By Proposition~\ref{lem:concen} with $w_s:=w$ for each $s\in [t]$ and $\lambda_t:= w t$, for any $\delta \in (0,1]$, it follows that
$$ \Pro \left[ \sum_{s=1}^t W_s \geq (1 + \delta)tw \right] <  e^{-\delta^2 w t /3}\;.$$ 
Using the definition of $W_s$ in~\eqref{eq:rescale}, we obtain
\begin{equation*} 
\Pro \left[ \sum_{s=1}^t Z_s \geq(1 + \delta) (y-a)t +a t \right] \leq  e^{-\frac{\delta^2 t (y-a)}{3(b-a)}}\;. 
\end{equation*}
Recall that $a< y< 0$. Choosing $\delta =-\frac{y}{2(y-a)}\in (0,1]$, we have 
$
(1 + \delta) (y-a) t + at = yt +\delta(y-a) t = \frac{yt}{2}.
$
Using that $b-a\geq y-a$, we obtain 
\begin{equation*} 
\Pro \left[ \sum_{s=1}^t Z_s \geq \frac{ty}{2} \right] < e^{-\frac{y^2}{12(b-a)^2}\cdot t}\;. 
\end{equation*}
Finally, note that 
\begin{align*} 
\sum_{s=1}^t  Z_s &=  \ind{t>\tau} ((-1)^m(Y_{\tau} - Y_0)+ (t- \tau )y) + \ind{t\leq \tau }(-1)^m(Y_t - Y_0) \nonumber\\
&= (-1)^m (Y_{\tau \wedge t} - Y_0) + \ind{t>\tau} (t-\tau) y
\end{align*}
and the lemma follows.
\end{proof}

\subsection{Proof of Proposition~\ref{prop:explo}}

In this subsection we will prove the Proposition~\ref{prop:explo}. 
We first need three more technical statements.

\begin{lemma}\label{lem: pre_concentration}
Suppose $n\in \N$ and $1/n\ll\alpha,\beta\ll \xi\ll \eta,\rho\ll \mu, 1/\avd ,p\leq1$. 
Suppose $S_0\subseteq V$ and let $\cD$ be a degree sequence on $V$ such that $\Td\leq\avd n$ and 
$w\in S_0$ for every $w\in V$ with $d(w)> n^{1/4}$  and $\sum_{v\in S_0} d(v)\leq \alpha n$. 
Let $\tau$ be the smallest $t\leq n$ such that either $X_t>\beta n$ or $M_t< (1-\eta)\Sigma^\cD$ - if this does not exist, we set $\tau = n+1$.
Conditional on the event that  $d_{H_0}(w)=d(w)$ for every $w\in V$ with $d(w)> n^{1/4}$, then 
$$ \Pro [ \tau\leq \xi n, X_{\tau} \leq \beta n ] = o(n^{-2}). $$ 
\end{lemma}
\begin{proof}
Observe that $ \Pro [ \tau\leq \xi n, X_{\tau} \leq \beta n, M_\tau\geq (1-\eta)\Td ] = 0 $.
Thus 
\begin{align} \label{eq:1st_upper_bound}
	\Pro [ \tau\leq \xi n, X_{\tau} \leq \beta n]
	&=\Pro [ \tau\leq \xi n, X_{\tau} \leq \beta n, M_\tau< (1-\eta)\Td ]\;. 
\end{align}
Note first that if $M_\tau<(1-\eta)\Td$, then
$\sum_{v\in S_\tau} d(v)\geq \eta \Td\geq \eta n$. 
Let $R_t$ be the set of times $s \in \{0,\dots,t\}$ where the edge $v_{s+1}w_{s+1}$ has percolated
and let $R_t'$ be the set of times where $s \in \{0,\dots,t\}$ where $X_s=0$.
Therefore, we have
\begin{align}\label{eq:sum}
\sum_{t\in R_\tau} ( \hat{d}_t(w_{t+1}) +d_{F_t}(w_{t+1}))+
\sum_{t\in R_\tau'} \hat{d}_t(w_{t+1})
\geq  \sum_{v\in S_\tau} d(v) -  \sum_{v\in S_0} d(v) \geq (\eta- \alpha) n\;. 
 \end{align}
At each step $s\in \{0,\ldots,t\}$ of the process at most one edge is added to $F_s$. For every $1\leq t\leq \tau$,  it follows that 
$$
\sum_{s\in R_t} d_{F_s}(w_{s+1}) \leq t+1\leq \xi n+1\;.
$$
Taking $\alpha \ll \xi\ll \eta$ and using~\eqref{eq:sum} one concludes
\begin{align} 
\sum_{t\in R_\tau\cup R_\tau'} \hat{d}_t(w_{t+1})
 \geq \left(\sum_{t\in R_\tau} \hat{d}_t(w_{t+1})+ d_{F_t}(w_{t+1}) \right)-\xi n-1 +\sum_{t\in R_\tau'} \hat{d}_t(w_{t+1})
 \geq \frac{\eta n}{2} \;. \label{eq:sum_hat_d} 
\end{align}
From~\eqref{eq:change} and~\eqref{eq:change2}, and with $\xi\ll \eta,p$, it follows that
\begin{eqnarray}\label{eq:previous}
X_{\tau} 
&=& X_{0}- (\tau-|R_\tau|)+\sum_{t\in R_\tau} ( \hat{d}_t(w_{t+1})-2 -2d'_t(w_{t+1}))+\sum_{t\in R_\tau'} \hat{d}_t(w_{t+1})\nonumber\\
&\geq& -3\tau+ \frac{1}{2}\sum_{t\in R_\tau \cup R_{\tau}'} \hat{d}_t(w_{t+1}) +\frac{1}{2}\left(\sum_{t\in R_\tau} \hat{d}_t(w_{t+1}) -4d'_t(w_{t+1})\right)\nonumber\\
&\stackrel{ (\ref{eq:sum_hat_d})}{\geq} &\frac{\eta n}{4} - 3\xi n + \frac{1}{2}\left(\sum_{t\in R_\tau} \hat{d}_t(w_{t+1}) -4	d'_t(w_{t+1})\right)\nonumber\\
& \geq &\frac{\eta n}{8}  +\frac{1}{2}\left(\sum_{t\in R_\tau} \hat{d}_t(w_{t+1}) -4 d'_t(w_{t+1})\right) \;.
\end{eqnarray}
Therefore, we deduce that 
\begin{eqnarray} 
\lefteqn{\Pro [ \tau\leq \xi n, X_{\tau} \leq \beta n, M_\tau< (1-\eta)\Td ] \leq} \nonumber \\ 
& &\Pro \left[ \tau\leq \xi n, X_{\tau} \leq \beta n, X_\tau \geq \frac{\eta n}{8}  +\frac{1}{2}\left(\sum_{t\in R_\tau} \hat{d}_t(w_{t+1}) -4 d'_t(w_{t+1})\right)  \right]. \label{eq:2nd_upper_bound}
\end{eqnarray}

We take $\beta \ll \eta$.
So in order that $X_\tau\leq \beta n$,
it suffices to prove that $-\sum_{t\in R_\tau} (\hat{d}_t(w_{t+1}) -4 d'_t(w_{t+1}))$ is not too large.

We define the following sequence $Y_1,\ldots,Y_{\xi n}$ of random variables.
Let $Y_0:=0$.
Suppose $t$ is the $s$-th smallest entry in $R_\tau$.
We set
$$
Y_{s}:= Y_{s-1} -(\hat{d}_t(w_{t+1}) -4 d'_t(w_{t+1}));
$$
in the case where $|R_\tau|<s$ and $s\leq \xi n$, we set
$$
Y_{s}:= Y_{s-1} -1.
$$
Observe that $|Y_{s} - Y_{s-1}|\leq 4n^{1/4}$.
Let $\{\cF_s\}_{s=0}^{\xi n}$ be the filtration induced by the sequence $\{Y_s\}_{s=0}^{\xi n}$.

Suppose again that $t$ is the $s$-th smallest entry in $R_\tau$.
We apply Lemma~\ref{lem: back edges} with  $\nu=\beta$, $Z'=S_t$, $H'=G^\cD[S_t]$, $F'=F_t$, $x=w_{t+1}$, $z=v_{t+1}$ and $\cE=\{xz\in E(G^\cD)\}$. 
The first three conditions of the lemma are satisfied: the first one is immediate from our hypothesis, the second one follows from $X_t\leq \beta n$ and the third one from the fact that $M_t-\sum_{w\in V\sm S_t} d_{F_t}(w)\geq (1-\eta) \Td-t\geq n/20$. 
Moreover, $xz\notin E(F')$ holds. Similarly as in~\eqref{eq:back_edges}, we obtain,
$$
\Exp[d'(w_{t+1})\mid \cH_t]\leq \frac{1}{10}\Exp[\hat{d}(w_{t+1})\mid \cH_t]\;.
$$
Let $t_{-1}$ be defined such that $t_{-1}-1$ is the $(s-1)$-th smallest entry in $R_\tau$ or $-1$ if $s=1$.
Observe that given $\cH_{t_{-1}}$
we still can apply Lemma~\ref{lem: back edges} to any possible input with history $\cH_t$.
This implies that 
$$
\Exp[d'(w_{t+1})\mid \cH_{t_{-1}}]\leq \frac{1}{10}\Exp[\hat{d}(w_{t+1})\mid \cH_{t_{-1}}]\;.
$$
Clearly $\mathbb{E}[\hat{d}_t(w_{t+1})|\cH_{t_{-1}}]\geq 1$, so
$$
\Exp[Y_{s+1}-Y_s| \cF_{s}]\leq -\frac{1}{2}\;.
$$

We can apply Lemma~\ref{lem:our_concentration} to the collection $Y_0,\ldots Y_{\xi n}$ with $-a=b=4n^{1/4}, y= -1/2, m=0$ and $t = \xi n$,
where $\tau$ is the stopping time $\tau=\xi n$,
to conclude that
$$
\Pro\left[-Y_{\xi n}< \frac{\xi n}{4}\right]=o(n^{-2}).
$$
Observe that by construction of $Y_s$ we have $\sum_{t\in R_{\tau\wedge \xi n}} \left(\hat{d}_t(w_{t+1}) -4 d'_t(w_{t+1})\right)\geq - Y_{\xi n}-\xi n$. 
Hence
\begin{equation} \label{eq:bad_e}
\Pro\left[\sum_{t\in R_{\tau\wedge \xi n}} \left(\hat{d}_t(w_{t+1}) -4 d'_t(w_{t+1})\right) 
< -\frac{3\xi n}{4} \right] 
= o(n^{-2})\;.
\end{equation}
So by~\eqref{eq:2nd_upper_bound} we can further bound 
\begin{eqnarray*} 
\lefteqn{\Pro [ \tau\leq \xi n, X_{\tau} \leq \beta n, M_\tau< (1-\eta)\Td ] \leq} \\ 
& &\Pro \left[ \tau\leq \xi n, X_{\tau} \leq \beta n, X_\tau \geq \frac{\eta n}{8}  -  \frac{3\xi n}{8}   \right]   
 + \Pro \left[ \sum_{t\in R_{\tau\wedge \xi n}} \left(\hat{d}_t(w_{t+1}) -4 d'_t(w_{t+1})\right) 
< -\frac{3\xi n}{4} \right] \\
&\stackrel{\eqref{eq:bad_e}}{\leq} &
\Pro \left[ \tau\leq \xi n, X_{\tau} \leq \beta n, X_\tau \geq \frac{\eta n}{8}  -  \frac{3\xi n}{8}   \right]    
+ o(n^{-2}).
\end{eqnarray*}
We will show that the first term is 0.
Indeed, 
$$
X_\tau>  \left( \frac{\eta}{8} - \frac{3\xi}{8}\right) n \geq \frac{\eta}{16}\cdot n > \beta n\; ;
$$
so this cannot hold simultaneously with $X_\tau \leq \beta n$. 
Finally by~\eqref{eq:1st_upper_bound}, we conclude that the probability that $\tau \leq \xi$ and 
$X_\tau \leq \beta n$ is $o(n^{-2})$.
\end{proof}

\begin{lemma}\label{lem: concentration}
Suppose $n\in \N$ and $1/n\ll\alpha\ll\beta\ll \eta,\rho\ll \mu,1/\avd ,p\leq1$. 
Let $S_0\subseteq V$ and let $\cD$ be a degree sequence on $V$ such that $\Td\leq\avd n$
and $w\in S_0$ for every $w\in V$ with $d(w)> n^{1/4}$,  and $\sum_{v\in S_0} d(v)\leq \alpha n$. 
Let $\tau$ be the smallest $t\leq n$ such that either $X_t>\beta n$ or $M_t< (1-\eta)\Sigma^\cD$ - if this does not exist, we set $\tau = n+1$.
Conditional on the event that  $d_{H_0}(w)=d(w)$ for every $w\in V$ with $d(w)> n^{1/4}$, the following holds:
\begin{enumerate}[label={\rm (\roman*)}]
\item If $\sum_{u\in V\sm S_0}d(u)(p(d(u)-1)-1)\leq -\mu n$  and $\tau_1$ is the smallest $t$ such that $X_{t}=0$, then the probability that $\tau_1 >\rho n$ is $o(1/n)$.

\item  If $R_p^{\cD} \geq \mu n$, then the probability that  $\tau >\rho n$ or that $X_\tau \leq\beta n$ is $o(1/n)$.
\end{enumerate}
\end{lemma}
\begin{proof}
Recall that $I(uv)$ is the indicator random variable that is equal to 1 if and only if $uv\in E(G^\cD)$ survives percolation when it is exposed. 
Also, recall~\eqref{eq:change}: if $X_t>0$, then
$$
X_{t+1}-X_t= -(1-I(v_{t+1}w_{t+1}))+I(v_{t+1}w_{t+1})((\hat{d}_t(w_{t+1})-2) - 2d_t'(w_{t+1}))\;.
$$

We first prove (i). 
Consider the sequence $Y_0, Y_1, \ldots$ of random variables such that $Y_0:=X_0$ and
$$
Y_s:= Y_{s-1}+\ind{s\leq \tau \wedge \tau_1} \left(X_{s}-X_{s-1}\right).
$$
Thus $|Y_s-Y_{s-1}|\leq 2n^{1/4}$.  
Let us take $\lambda$ satisfying $\eta,\rho\ll \lambda \ll \mu,1/\avd,p$. 
Since $X_{s-1}>0$ if $s\leq \tau\wedge \tau_1$, by Lemma~\ref{lem: exp change upper}, we have
$$
\Exp[Y_s-Y_{s-1}| \cH_{s-1}]\leq -\lambda \ind{s\leq \tau\wedge \tau_1}=:y\ind{s\leq \tau\wedge \tau_1}\;.
$$ 
Let $\nu$ and $\xi$ be such that  $\alpha\ll \nu\ll\beta \ll \xi \ll \eta,\rho$.
We now apply Lemma~\ref{lem:our_concentration} to $Y_s$ with $-a=b=2n^{1/4}$, $m=0$ and $t = c n$, for some $c$ such that $1/n \ll c<1$, to conclude that 
\begin{align} \label{eq:upper_tail} 
\Pro \left[ (X_{\tau\wedge \tau_1 \wedge c n} - X_0)- \ind{c n > \tau\wedge \tau_1} (c n- \tau\wedge \tau_1 )\lambda \leq -\frac{\lambda}{2}\cdot c n \right]
&\geq
1 -e^{-\frac{\lambda^2}{48}\cdot c n^{1/2}}\notag\\
&\geq  1-n^{-2}. 
\end{align}
Let $\mathcal{E}_1(t)$ denote the event $ (X_{\tau\wedge \tau_1\wedge t} - X_0)- \ind{t > \tau\wedge \tau_1} (t- \tau\wedge \tau_1 )\lambda \leq -\frac{\lambda}{2}\cdot t$.
We use Lemma~\ref{lem: pre_concentration} (for the second inequality below) and we write
\begin{eqnarray*}
\Pro [\tau_1 > \rho n] &=& \Pro[\tau_1 > \rho n, \tau \leq \xi n] + \Pro[\tau_1 > \rho n, \tau > \xi n] \\
& \leq& \Pro[\tau_1 > \rho n, \tau \leq \xi n, X_{\tau} > \beta n] + n^{-2} + \Pro[\tau_1 > \rho n, \tau > \xi n] \\ 
&\leq & \Pro[\tau_1 > \rho n, \tau \leq \xi n, X_{\tau} > \beta n, \mathcal{E}_1(\nu n)] + \Pro [\overline{ \mathcal{E}_1(\nu n)}] + n^{-2} + \Pro[\tau_1 > \rho n, \tau > \xi n] \\
&\stackrel{(\ref{eq:upper_tail})}{=}& \Pro[\tau_1 > \rho n, \tau \leq \xi n, X_{\tau} > \beta n, \mathcal{E}_1(\nu n)] + 2n^{-2} + \Pro[\tau_1 > \rho n, \tau > \xi n].
\end{eqnarray*}
Suppose that the events $\tau_1>\rho n, \frac{3\nu}{4} n\leq \tau \leq \xi n, X_{\tau} > \beta n$ and $\mathcal{E}_1(\nu n)$ are realised simultaneously. 
Recall that  $\alpha\ll \nu \ll\beta \ll \xi \ll \rho\ll \lambda$. 
Since $\tau_1>\rho n$, we have $X_{\tau\wedge \tau_1 \wedge \nu n}>0$.
Then, we reach a contradiction in the following way
$$
0< X_{\tau\wedge \tau_1 \wedge \nu n} \leq -\frac{\lambda}{2}\nu n  + X_0 +\ind{\nu n>\tau\wedge \tau_1}(\nu n- \tau\wedge \tau_1)\lambda\leq -\frac{\lambda\nu}{8} n  <0\;.
$$
Suppose that the events $\tau_1 > \rho n$, $\tau \leq \frac{3\nu}{4} n, X_{\tau} > \beta n$ and $\mathcal{E}_1(\nu n)$ are realised simultaneously. 
Again, we reach a contradiction as follows
$$
\beta n< X_{\tau \wedge \tau_1\wedge \nu n} = X_{\tau} \leq -\frac{\lambda}{2}\nu n  + X_0 +(\nu n- \tau)\lambda\leq 2\lambda\nu n  < \beta n\;.
$$
Hence, $\Pro[\tau_1 > \rho n,\tau \leq \xi n, X_{\tau} > \beta n, \mathcal{E}_1(\nu n)] =0$.

Thereby, 
\begin{eqnarray*}
\Pro [\tau_1 > \rho n] &\leq&  \Pro[\tau_1 > \rho n, \tau > \xi n] + 2n^{-2} \\ 
&\leq & \Pro[\tau_1 > \rho n, \tau > \xi n,\mathcal{E}_1(\xi n)] + \Pro [\overline{ \mathcal{E}_1(\xi n)}] + 2n^{-2} \\
& \stackrel{(\ref{eq:upper_tail})}{=}& \Pro[\tau_1 > \rho n, \tau > \xi n,\mathcal{E}_1(\xi n)] + 3n^{-2}.
\end{eqnarray*}
But again the event $\tau_1 > \rho n, \tau > \xi n$ cannot occur simultaneously with $\mathcal{E}_1(\xi n)$, since otherwise,
$$
X_{\xi n} =X_{\tau\wedge \tau_1\wedge \xi n} \leq -\frac{\lambda}{2}\xi n+ X_0 <0\;.
$$
We conclude that $ \Pro [\tau_1 > \rho n] \leq 3n^{-2}.$
\medskip


\noindent
We proceed to prove (ii). 
Let $Y_0:=0$. 
For $s\geq 1$, 
consider the random variable 
$$
Y_s:= Y_{s-1}-\ind{s\leq \tau} \left(X_{s}-X_{s-1}\right)\;.
$$
By the second part of Lemma~\ref{lem: back edges2},
$$
\Exp[Y_s| \cH_{s-1}]\leq -\lambda\ind{s\leq \tau}=:y\ind{s\leq \tau}\;.
$$ 
Let us take $\xi$ and $\lambda$ such that $\beta \ll \xi \ll \eta,\rho\ll \lambda \ll \mu,1/\avd,p$.
Similarly as before, we can apply Lemma~\ref{lem:our_concentration} to the random variables $Y_s$ with $-a=b=2n^{1/4}$, $m=1$ and 
$t =\xi n$ to conclude that 
\begin{equation} \label{eq:upper_tail2} 
\Pro \left[ (X_{\tau\wedge \xi n} - X_0)+  \ind{\xi n > \tau} (\xi n -  \tau )\lambda>\frac{\lambda}{2}\cdot \xi n \right] \geq
1 -e^{-\frac{\lambda^2}{48}\cdot \xi n^{1/2}} = 1-o(n^{-2}). 
\end{equation}
Now, let $\mathcal{E}_2(\xi n)$ denote the event $X_{\tau\wedge \xi n} - X_0+  \ind{\xi n > \tau} (\xi n -  \tau )\lambda>\frac{\lambda}{2}\cdot \xi n$. 

Letting $\xi \ll \rho$, we then have 
\begin{eqnarray*}
\Pro [\tau > \rho n ]  \leq   \Pro [\tau > \xi n ]
& \leq  & \Pro [\tau > \xi n, \mathcal{E}_2(\xi n) ] + \Pro[\overline{\mathcal{E}_2(\xi n) }]\\
& \stackrel{(\ref{eq:upper_tail2})}{\leq}  & \Pro [\tau > \xi n, \mathcal{E}_2(\xi n) ] + n^{-2}\;.
\end{eqnarray*}
But if $ \tau > \xi n$ holds simultaneously with $ \mathcal{E}_2(\xi n)$, then 
we have $X_{\xi n}=X_{\tau \wedge \xi n} > X_0 + \frac{\lambda\xi}{2}\cdot n \geq \beta n$ which contradicts that $\tau > \xi n$.   
So, this event has probability 0. 

Similarly, using Lemma~\ref{lem: pre_concentration} (for the first inequality) we obtain
\begin{eqnarray*}
\Pro [X_{\tau} < \beta n ] &=& \Pro [X_{\tau} < \beta n, \tau \leq \xi n] +  \Pro [X_{\tau} < \beta n, \tau > \xi n] \\
&\leq& n^{-2} + \Pro [X_{\tau} < \beta n, \tau > \xi n]  \\
&\leq& n^{-2} + \Pro [X_{\tau} < \beta n, \tau > \xi n, \mathcal{E}_2(\xi n)] + \Pro [\overline{\mathcal{E}_2(\xi n)}] \\
&\stackrel{(\ref{eq:upper_tail2})}{=}& n^{-2} + \Pro [X_{\tau} < \beta n, \tau > \xi n, \mathcal{E}_2(\xi n)].
\end{eqnarray*}
As above, $ \tau > \xi n$ and $ \mathcal{E}_2(\xi n)$ are incompatible. 
Therefore, the second event has probability 0, whereby we deduce that $\Pro[X_\tau<\beta n]=o(1/n)$.
\end{proof}

\begin{lemma}\label{lem: edges to vertices}
Let $1/n\ll\nu, \rho\ll 1/\avd ,p\leq1$. 
Suppose that $S\subseteq V$ with $|S|\leq \rho n+1$ and $U\subseteq S$.
Suppose $H$ is a graph with vertex set $S$ and $F$ is a bipartite graph with vertex partition $(S,V\sm S)$ and $|E(F)|\leq \rho n$.
Let $\cD$ be a degree sequence on $V$ such that $\Td\leq\avd n$ and, moreover,
$d(w)=d_H(w)$ for every $w\in V$ with $d(w)\geq n^{1/4}$ and
$\sum_{u\in U}(d(u)-d_H(u)-d_F(u) ) \geq \nu n$.

Conditional on $G^\cD[S]=H$ and on $F\subseteq G^\cD$ being the set of edges between $S$ and $V\sm S$ that have failed to percolate in $G_p^\cD$,
the probability that the union of components of $G^\cD_p$ that intersect $U$ contains at most $(\nu p/(20\avd))n $ vertices is $o(1)$.
\end{lemma}
\begin{proof}
We may assume that $|U|< (\nu p/(20\avd)) n$. 
Let $\hat{G}^\cD:=G^\cD-E(F)$. 
Our aim is to show that $N_{\hat{G}^\cD}(U)\subseteq N_{G^\cD}(U)$ is typically large.  
Note that for every vertex $w\in N_{\hat{G}^\cD}(U)$, 
there is at least one edge $uw\in E(\hat{G}^\cD)$ with $u\in U$ that has not been exposed to percolation. 
In the second part of the proof, we will show that many of these edges are preserved in $G^\cD_p$, implying that the union of components of $G_p^\cD$ that intersect $U$  contains many vertices.

Let $K:=\lfloor (\nu /5 \avd)n\rfloor$. For every $k<K$, let $\cF_k$ be the set of graphs $G$ with degree sequence $\cD$ such that $G[S]=H$, $F\subseteq G$ and  $|N_{\hat{G}}(U)|=k$, where $\hat{G}:=G- E(F)$. 
In order to estimate the probability of each $\cF_k$ we only use edges not contained in $E(H)\cup E(F)$ for a switch.

Consider a graph in $\cF_k$. There are at least $\nu n-k\geq 4\nu n/5$ choices for an edge $uv\in E(\hat{G})$ with $u\in U$, 
$v\in N_{\hat{G}}(U)$ and such that there exists $u'\neq u$ with $u'\in U$ and $u'v\in E(\hat{G})$.
Since $\delta(G)\geq 1$ and $d(w)\leq n^{1/4}$ for every $w\in \{u\}\cup N_{\hat{G}}(U)$, there are 
at least $(n-|S\cup N_{\hat{G}}(U)|)/2-|E(F)| -2n^{1/2}\geq n/3$ edges $xy\in E(\hat{G})$ with $x\notin S\cup N_{\hat{G}}(U)$ and which are in distance at least $2$ from $uv$. 
Thus, the total number of such switches into a graph in $\cF_{k+1}$ is at least $\nu n^2/4$.

Given a graph in $\cF_{k+1}$, then there are at most $(k+1)\avd n$ switches that transform it into a graph in $\cF_k$.

Thus, for every $k< K$, we have
\begin{align*}
\Pro[\cF_k]&\leq \frac{(k+1)\avd n}{\nu n^2 /4}\cdot\Pro[\cF_{k+1}]\leq \frac{4}{5}\cdot \Pro[\cF_{k+1}]\;,
\end{align*}
which implies
$$
\Pro[\cup_{k\leq K/2} \cF_k]\leq \Pro[\cF_{K/2}]\sum_{i\geq 0} (4/5)^{-i}\leq (4/5)^{- K/2+1} \Pro[\cF_{K}]=o(1)\;.
$$
That is, with probability $1-o(1)$, there are at least $(\nu/10 \avd)n$ vertices that are connected to $U$ by at least one edge in $\hat{G}^\cD$. These edges have still not been exposed for percolation. Chernoff's inequality (Lemma~\ref{lem:Chernoff}) now implies that with probability $1-o(1)$ a proportion of at least $p/2$ of them will be retained in $G^\cD_p$. Therefore, with probability $o(1)$, we have $|N_{G^\cD_p}(U)|\leq (\nu p/(20\avd)) n$. The conclusion follows.
\end{proof}

\begin{proof}[Proof of Propostion~\ref{prop:explo}]
We start with the first statement.
Suppose there exists a set $S\subseteq V$ such that $d(v)\leq n^{1/4}$ for every $v\notin S$, 
and for every  possible choice of $G$ with degree sequence $\cD$, we have
$\sum_{u\in N_G[S]}d(u)\leq \alpha n$
and $\sum_{u\in V\sm N_G[S]}d(u)(p(d(u)-1)-1)\leq -\mu n$.

We show that every vertex $u\in V$ is in a component of size at least $\gamma n$ with probability $o(1/n)$.
A union bound over all vertices completes the proof.

Suppose $u\in V$. 
We prove the desired statement conditional on every possible neighbourhood of $S$.
Thus let $S_0:=N[S]\cup \{u\}$ for some choice of $N[S]$.
Hence $\sum_{u\in S_0}d(u)\leq 2\alpha n $ and $\sum_{u\in V\sm S_0}d(u)(p(d(u)-1)-1)\leq -\mu n/2$.
Moreover,
for every vertex $v\in V$ with $d(v)>n^{1/4}$,
all its neighbours belong to $S_0$.
We apply the first part of Lemma~\ref{lem: concentration} with $\rho = \gamma /2$.
Since $\gamma\ll \mu$,  there exists a 
$t\leq \gamma n/2$ such that $X_t=0$ with probability $1-o(1/n)$. 
Since 
$|S_t|\leq t + |S_0| \leq (\gamma/2  + 2\alpha) n<\gamma n$, 
the union of all components that intersect $\{u\} \cup N[S]$ contain less than $\gamma n$ vertices with probability $1-o(1/n)$. 
\medskip

Now we prove the second statement.
Recall that now $\Delta(\cD)\leq n^{1/4}$.
Let $S_0:=\{u_0\}$ for an arbitrary vertex $u_0\in V$. 
Clearly,  $\sum_{v\in S_0}d(v)\leq \alpha n$.
Recall that $X_t$ counts the number of edges between $S_t$ and $V\sm S_t$ in the graph $G^\cD$ that have not yet been exposed for percolation. 
Observe that all the edges counted by $X_t$ will belong to the same component of $G^\cD_p$ if they survive percolation. Note that this component may not contain $u_0$. 

We choose $\beta$ and $\rho$ such that $\gamma\ll\beta \ll \rho \ll\mu$. 
By the second part of Lemma~\ref{lem: concentration}, with probability $1-o(1)$, 
there exists a $\tau\leq \rho n$ with $X_\tau\geq \beta n$. Recall that $\cH_\tau$ denotes the history of the exploration process, with the corresponding choice of $S_\tau$, $H_\tau$ and 
$F_\tau$ at time $\tau$. Let $U$ be the set of vertices from the component of $G^\cD_p$ under exploration at time $\tau$ that have been already explored; that is, the ones in $S_\tau$. Then 
$\sum_{u\in U} (d(u)-d_{H_\tau}(u)-d_{F_\tau}(u))=X_\tau \geq \beta n$. 
Moreover, $|S_\tau| \leq \tau +1\leq \rho n+1$ and $|E(F_\tau)| \leq \tau \leq \rho n$.
By Lemma~\ref{lem: edges to vertices} with $\nu=\beta$, $S=S_\tau$, $H=H_\tau$ and 
$F=F_\tau$, with probability $1-o(1)$, there exists a component in $G^\cD_p$ with at least $(\beta p/ (20 \avd))n\geq \gamma n$ vertices.
\end{proof}




\section{Degree sequences with many vertices of high degree} \label{sec:very_rob}

In this section we prove Proposition~\ref{prop:very_rob}.
As in the proof of Theorem~\ref{thm:rob},
we adapt our argumentation according to the structure of the degree sequence.
If not stated otherwise, we always consider a degree sequence $\cD$ on $V$
with average degree at most $\avd$, where $V$ is a set of size $n$. Recall that $\avd$ is assumed to be fixed.
In addition, we assume $1/n \ll 1/\avd\leq 1$.
We start with some notation, which we use throughout this section.
Let 
\begin{align*}
T&:=\{u\in V: d(u)\leq 3 \avd\}, \\
S_1&:=\{u\in V: d(u)\geq \log^2 n\},\\
S_2&:=\{u\in V: d(u)\geq n^{1/3}\}, \text{ and}\\
S_3&:=\{u\in V: d(u)\geq  n^{4/5}\}.
\end{align*}
We say $\cD$ satisfies $(D^1_\epsilon)$ if
\begin{align}\tag{$D^1_\epsilon$}\label{eq:De1}
	\sum_{u\in S_1} d(u)\geq \epsilon n\;,
\end{align}
and we say it satisfies $(D^3_\epsilon)$ if
\begin{align}\tag{$D^3_\epsilon$}\label{eq:De}
	\sum_{u\in S_3} d(u)\geq \frac{\epsilon n}{10}\;.
\end{align}

\subsection{Degree sequences with vertices of very high degree}

In this subsection we consider degree sequences $\cD$ that satisfy $(D^3_\epsilon)$.
We collect several results about such degree sequences, which we will use in the proof of Proposition~\ref{prop:very_rob}.


The first lemma shows that $G^\cD[S_3]$ is typically a clique.
\begin{lemma}\label{lem:clique}
Suppose $n\in \N$ and $1/n\ll 1/\avd \leq 1$. 
Let $V$ be a set of size $n$ and let
$\cD$ be a degree sequence on $V$ with $\Td\leq \avd n$.
Then the probability that $G^\cD[S_3]$ is a clique is at least $1-n^{-1/11}$.
\end{lemma}
\begin{proof}
Since $\Td \leq \avd n$, it follows that $|S_3|\leq \avd n^{1/5}$.
If $\Pro[uv\notin E(G^\cD)]\leq n^{-1/2}$ for every $u,v\in S_3$, 
a union bound over all pairs $u,v\in S_3$ proves the lemma.

It remains to prove that for each pair $u,v\in S_3$, 
we have $\Pro[uv\notin E(G^\cD)]\leq n^{-1/2}$. 
Let $\cF^-$ be the set of graphs $G$ on $V$ with degree sequence $\cD$ and $uv\notin E(G)$ and 
let $\cF^+$ be the set of graphs $G$ on $V$  with degree sequence $\cD$ and $uv\in E(G)$. 

Suppose $G\in \cF^-$. Since $d(u),d(v)\geq n^{4/5}$ and $\Td\leq \avd n$, there exist at least $n^{8/5}/2$ ordered pairs $(x,y)$ with $x\in N(u)$, $y\in N(v)$ and $xy\notin E(G)$. 
Switching $ux$ and $vy$ transforms $G$ into a graph in $\cF^+$. 

Suppose $G\in \cF^+$, then there are at most $\avd n$ switches that transform $G$ into a graph in $\cF^-$. 
Therefore, by (\ref{eq:fund}), we obtain
\begin{align*}
	\Pro[\cF^-]\leq \frac{2\avd n}{n^{8/5}}\cdot\Pro[\cF^+]\leq n^{-1/2}.
\end{align*}
\end{proof}

Conditional on $G^\cD[S_3]$ being a clique, $G^\cD_p[S_3]$ is a binomial random graph on $|S_3|$ vertices and edge probability $p$. 
Since $p\in (0,1)$, 
it is an exercise to check that $S_3$ induces a connected graph in $G^\cD_p$ with probability at least $1-c_1^{|S_3|}$, for some $c_1=c_1(p)<1$.
Together with Lemma~\ref{lem:clique} we obtain the following corollary.
\begin{corollary}\label{cor:A}
Suppose $n\in \N$ and $1/n\ll 1-c\ll p, 1/\avd\leq 1$. 
Let $V$ be a set of size $n$ and let
$\cD$ be a degree sequence on $V$ with $\Td\leq \avd n$.
Then
\begin{align*}
	\Pro[G^\cD_p[S_3]\ \text{\rm is disconnected}]\leq c^{|S_3|}.
\end{align*}
\end{corollary}

The next lemma shows that, typically, the vertices in $S_2\sm S_3$ are connected to a vertex in $S_3$ in $G^\cD$ if $|S_3|\geq 100$.
\begin{lemma}\label{lem:connectedtoS1}
Suppose $n\in \N$ and $1/n\ll 1/\avd \leq 1$.  
Let $V$ be a set of size $n$ and let 
$\cD$ be a degree sequence on $V$ with $\Td\leq \avd n$.
Assume that $|S_3|\geq 100$.
Then, with probability at most $1/n$,
there is a vertex $u\in S_2\sm S_3$ which is not adjacent to a vertex in $S_3$.
\end{lemma}
\begin{proof}
It suffices to show that every vertex $u\in S_2$
is adjacent to a vertex in $S_3$ with probability at least $1-n^{-2}$.
Let $u\in S_2$ and let $0\leq k\leq 50$.
Let $\cF_k$ be the event that $u$ is adjacent to exactly $k$ vertices in $S_3$.

Consider a graph $G\in\cF_{k+1}$.
Clearly, there are at most $(k+1)\avd n$ switches transforming $G$ into a graph in $\cF_k$.

Consider a graph $G\in\cF_k$.
Let $x$ be any vertex in $S_3$ which is not adjacent to $u$ (since $|S_3| \geq 100$ but $k\leq 50$ there is such a vertex).
Thus there are at least $n^{1/3+4/5}=n^{17/15}$ pairs $(v,y)$ such that $v\in N(u)$ and $y\in N(x)$.
For at most $n$ pairs $v=y$ and for at most $2\avd n$ pairs, we have $vy\in E(G)$.
Thus at least $n^{17/15}/2$ pairs lead to a $\{ uv,xy\}$-switch transforming $G$ into a graph in $\cF_{k+1}$.
Hence
\begin{align*}
	\Pro[\cF_k]\leq n^{-1/15}\cdot \Pro[\cF_{k+1}].
\end{align*}
Moreover, this implies
\begin{align*}
	\Pro[\cF_0]\leq n^{-50/15} \cdot\Pro[\cF_{50}]\leq n^{-2},
\end{align*}
which completes the proof.
\end{proof}
Recall that $T$ is the set of vertices of degree at most $3\avd$.
As $\cD$ has average degree at most $\avd$,
many vertices belong to $T$.
More precisely,
as every vertex in $V \sm T$ has degree at least $3\avd$ and the average degree at most $\avd$,
we conclude $|V \sm T|\leq n/3$.
Thus
\begin{align}\label{eq:sizeT}
	|T|\geq \frac{2n}{3}\,.
\end{align}

The next lemma shows that many vertices in $T$ are adjacent to a vertex in $S_2$ if \eqref{eq:De} holds.

\begin{lemma}\label{lem:N(S_2)T}
Suppose $n\in \N$ and $1/n\ll 1000\epsilon\leq 1/\avd \leq  1$.  
Suppose $V$ is a set of size $n$ and 
$\cD$ is a degree sequence on $V$ with $\Td\leq \avd n$ that satisfies \eqref{eq:De}.
Then, 
\begin{align*}
	\Pro[|N(S_2)\cap T|\leq \epsilon^2n]\leq n^{-1}\;.
\end{align*}
\end{lemma}
\begin{proof}
For every $0\leq k\leq 2\epsilon^2n$,
let $\cF_k$ be the set of graphs with degree sequence $\cD$ such that $|T\cap N(S_2)|=k$.

Suppose $0\leq k \leq 2\epsilon^2n$.
Consider at graph $G\in \cF_{k+1}$.
In order to transform $G$ into a graph in $\cFk$,
we need to select a vertex $u\in T\cap N(S_2)$ which has at exactly one neighbour $v$ in $S_2$.
Then there are at most $\avd n$ switches involving $uv$.
Thus in total, there are at most $2\epsilon^2\avd n^2$ switches from $\cFkk$ to $\cFk$.

Suppose $G\in\cFk$. 
Recall that $k\leq 2\epsilon^2 n$.
Since $\Td\leq\avd n$, we have $|S_2|\leq \avd n^{2/3}$ and $|S_3|\leq \avd n^{1/5}$. As \eqref{eq:De} holds and as there at most $(\avd)^2 n^{2/3+1/5}\leq \epsilon n/40$ edges between the vertices of $S_3$ and $S_2$, it turns out that there are at least $\epsilon n/15$ edges $xy$ such that $x\in S_3\subseteq S_2$ and $y\in N(S_2)$. 
More specifically, since 
$k\leq 2 \epsilon^2 n$, there are least $\epsilon n/20$ edges $xy$ with $x \in S_3$ and such that $y$ satisfies one of the following: either
$y\in N(S_2)\setminus T$ or if $y \in N(S_2)\cap T$, then it has at least two neighbours in $S_2$. 
Fix such a choice of an edge $xy$. 
Note that if we switch $xy$ with another edge $uv$ such that $u\in T\sm N(S_2)$, we can only increase the neighbourhood of $S_2$.
Observe that there are at most $n^{2/3}$ edges $uv$ such that $u\in T\sm N(S_2)$ and $v\in N(y)$.
Furthermore, $|T\sm N(S_2)|\geq n/2$.
Let $u\in T\sm N(S_2)$ and $v\in N(u)$ such that $v\notin N(y)$.
Then there are at least $n/2-n^{2/3}\geq n/3$ choices for the edge $uv$.
Observe that the $\{uv,xy\}$-switch  yields a graph in $\cFkk$ and there are at least $\epsilon n^2/60$ switches from $\cFk$ to $\cFkk$.
Hence
\begin{align*}
	\Pro[\cF_k]\leq \frac{120\epsilon^2 \avd n^2}{\epsilon n^2}\cdot \Pro[\cF_{k+1}]\leq \frac{1}{2}\Pro[\cF_{k+1}].
\end{align*}
In particular,
\begin{align*}
	\sum_{k=0}^{\epsilon^2 n}\Pro[\cF_k]\leq 2 \Pro[\cF_{\epsilon^2 n}] \leq 2^{-\epsilon^2 n+1}\,,
\end{align*}
which completes the proof.
\end{proof}

\subsection{Lighter degree sequences}

In this subsection we consider degree sequences that satisfy $\eqref{eq:De1}$  but not \eqref{eq:De}.
The first lemma shows that in this case, typically, the minimum degree of $G^\cD[S_1]$ is large.
\begin{lemma}\label{lem:large min deg}
Suppose $n\in \N$ and $1/n\ll\epsilon \ll 1/\avd \leq 1$.  
Let $V$ be a set of size $n$ and let $\cD$ be a degree sequence on $V$ with $\Td\leq \avd n$.
Assume also $\cD$ that satisfies~\eqref{eq:De1}, but not \eqref{eq:De}. 
Then, with probability $o(1)$, there exists a vertex $u\in S_1$ such that $d_{G^\cD[S_1]}(u)\leq \min\{d(u),n^{1/6}\}\cdot {\epsilon}/{(16\avd)}$.
\end{lemma}
\begin{proof}
Let $u\in S_1$ and let $K :=\lfloor \min\{d(u),n^{1/6}\}\cdot {\epsilon}/{(16\avd)}\rfloor$.  
For every $0\leq k\leq 2K$, 
let $\cF_k$ be the set of graphs $G$ with degree sequence $\cD$ such that $d_{G[S_1]}(u)=k$. 

Suppose $G\in \cF_k$. 
There are at least $d(u)-k\geq d(u)/2$ choices for an edge $uv$ with $v\in N(u)\sm S_1$. 
The degree of $v$ is less than $\log^2 n$ and each one of its neighbours has either degree less than $n^{4/5}$ (outside $S_3$) 
or it belongs to $S_3$. The former have total degree less than $n^{4/5} \log^2 n$, whereas the latter have total degree at most 
$\epsilon n/10$ (since \eqref{eq:De} does not hold).
Hence, there are at most $n^{4/5}\log^2 n + \epsilon n/10\leq \epsilon n/5$ edges at distance $2$ from $v$. 
Similarly, there are at most $n^{4/5} k + \epsilon n/10 \leq \epsilon n/5$ edges with one endpoint in $N(u)\cap S_1$. 
Since $(D_\epsilon^1)$ holds, there are at least $\epsilon n- 2\epsilon n/5 \geq \epsilon n/2$ edges $xy$ with $x\in S_1\sm N(u)$ and $y\notin N(v)$.
Performing a $\{ xy, uv\}$-switch,  we obtain a graph in $\cF_{k+1}$.
We conclude that there are at least $\epsilon d(u) n/4$ switches that transform $G$ into a graph 
in $\cF_{k+1}$. 

If $G$ is in $\cF_{k+1}$, there are at most $(k+1)\cdot \avd n$ switches that transform it into a graph in $\cF_{k}$. 
Therefore, for every $0\leq k< 2K$, we obtain
$$
\Pro[\cF_{k}]\leq 
\frac{4(k+1)\avd  n}{\epsilon d(u) n}\cdot\Pro[\cF_{k+1}]
\leq \frac{1}{2}\cdot \Pro[\cF_{k+1}]\;.
$$
Since $u\in S_1$, 
we have $d(u)\geq \log^2{n}$ and we obtain
$$
\sum_{k=0}^{K-1}\Pro[ \cFk] \leq 2^{-K} \Pro[\cF_{2K}]\leq n^{-2}\;.
$$
A union bound over all vertices $u\in S_1$ completes the proof.
\end{proof}

\begin{lemma}\label{lem:connec2}
Suppose $n\in \N$ and $1/n\ll p, 1/\avd\leq 1$. 
Let $V$ be a set of size $n$ and let $\cD$ be a degree sequence on $V$ with $\Td\leq \avd n$. 
For $R\subseteq V$ the following holds.
Conditional on $\delta(G^\cD [R])\geq 200\log{n}/p$, 
the probability that $G^\cD_p [R]$ is connected is $1-o(1)$. 
\end{lemma}
\begin{proof}
Let $N:=|R|$. 
Our proof strategy is to show that with high probability for every possible partition $(A,B)$ of $R$, 
there are edges between $A$ and $B$ in $G^\cD_p$.

Let $(A,B)$ be a partition of $R$ such that $\alpha:=|A|/N$ and $\alpha\leq 1/2$.
Let $K:=\lfloor 2\alpha N \log n/p\rfloor$. 
For every $0\leq k \leq 2K$, 
let $\cF_k$ be the set of graphs $G$ with degree sequence $\cD$ such that $\delta(G[R])\geq 200\log{n}/p$ and 
there are exactly $k$ edges between $A$ and $B$. 
In order to give an upper bound on $\Pro[\cF_k]$, 
we will consider switches between $\cFk$ and $\cFkkk$.

Let $G\in \cF_k$. 
We claim that there exist two subsets $A'\subseteq A$ and $B'\subseteq B$ with $|A'|\geq |A|/2$ and $|B'|\geq |B|/2$ 
such that for every 
$u\in A'$ (and every $y\in B'$), 
there are at least $100\log{n}/p$ edges from $u$ to $A$ (and from $y$ to $B$). 
We prove this claim for $A$, 
because the latter case is similar. 
Our assumption is that $0\leq k\leq 2K\leq 4\alpha N \log n/p$ and $\delta(G[R])\geq200\log n/p$. 
Let $A''\subseteq A$ be the subset that consists of all those vertices $u$ such that $d_{G[A]}(u)< 100 \log n/p$. 
If $|A''|\geq |A|/2 = {\alpha N}/2$, 
then 
\begin{align*}
	e(A,B)
	> \frac{\alpha N}{2} \cdot \frac{100\log n}{p} 
	> \frac{4 \alpha N \log n}{p}
	\geq 2K\geq k,
\end{align*}
which is a contradiction. 
Therefore $|A''| < |A|/2$, and setting $A' = A\sm A''$ we have $|A'| = |A\sm A''| \geq |A|/2$. Similarly, we set $B'=B\sm B''$.

Next we claim that the edges of $G[A]$ can be oriented in such a way so that 
every vertex in $A'$ has out-degree at least $48 \log{n}/p$ in $A$. 
To obtain such an orientation, 
start consistently orienting the edges of undirected cycles in $G[A]$ until the undirected graph induces a forest. 
Afterwards iteratively and consistently orient maximal undirected paths in this forest. 
If so, the out-degree of a vertex in $A$ is at least the in-degree minus $1$. 
Since $d_{G[A]}(u)\geq 100\log{n}/p$ for every vertex $u\in A'$,
the vertex $u$ has at least $50\log{n}/p-1\geq 48 \log n/p$ out-neighbours.
Similarly, one can also orient the edges of $G[B]$ in such a way that every vertex in $B'$ has out-degree at least $48\log{n}/p$ in $B$.

For each vertex in $u\in A'$, 
select a set $E(u)$ of exactly $48\log{n}/p$ directed edges from $u$ to a vertex in $A$, and analogously select $E(y)$, for each $y\in B'$. 
We will only count the (possible) $\{uv,xy\}$-switches 
with $u\in A'$, $y\in B'$, $uv\in E(u)$ and $yx\in E(y)$. 
For the switch to be valid, we insist on $ux,vy\notin E(G)$. 
Each edge $ab$ with $a\in A,b\in B$
can only invalidate switches of the form $\{av,by\}$ with $(a,v)\in E(a)$ and of the form $\{ua,xb\}$ with $(b,x)\in E(b)$;
that is, one edge $ab$ invalidates at most 
\begin{align*}
	|E(a)|\cdot (1-\alpha)N+|E(b)|\cdot \alpha N
	= \frac{48N\log n}{p}
\end{align*}
switches.
Hence in total the edges between $A$ and $B$ block at most
\begin{align*}
	2K \cdot \frac{48N\log n}{p} \leq \frac{192\alpha N^2 \log^2 n}{p^2}
\end{align*}
possible switches.
Recall that $|A'|\geq |A|/2$, $|B'|\geq |B|/2$.
Thus, by using $1-\alpha\geq 1/2$, there are at least 
$$
\frac{\alpha N}{2}\cdot \frac{48\log n}{p} \cdot \frac{(1-\alpha) N}{2}\cdot \frac{48\log n}{p}
- \frac{192\alpha N^2 \log^2 n}{p^2}
\geq \frac{96\alpha N^2 \log^2 n}{p^2}\;,
$$ 
switches that transform the graph $G$ into a graph in $\cFkkk$.
Since we only switch edges with both endpoints in $R$, 
the minimum degree in the graph induced by $R$ stays the same.

Consider a graph in $\cFkkk$.
Clearly,
there are at most $(k+2)^2$ switches that transform $G$ into a graph in $\cF_k$. 
Therefore, for every $0\leq k \leq 2K-2$, we conclude
$$
\Pro[\cF_k]
\leq \frac{p^2(k+2)^2}{96\alpha N^2 \log^2 n}
\cdot\Pro[\cFkkk] 
\leq \frac{1}{4}\cdot \Pro[\cFkkk] \;.
$$
We conclude that the probability there are less than $K$ edges in $G^\cD$ between $A$ and $B$ is small, namely
\begin{align}\label{eq:inter}
\sum_{k=0}^{K-1}\Pro[ \cFk]
\leq 2\cdot 4^{-K/2} (\Pro[\cF_{2K}]+\Pro[\cF_{2K-1}])
\leq 2^{-K+1}\;.
\end{align}
If $k\geq K$, 
then $\Pro[e(G_p[A,B])=0\mid \cFk]\leq (1-p)^{K}$. 
Therefore, provided $\delta(G[R])\geq200\log{n}/p$, 
the probability that $e(G^\cD_p[A,B])=0$ is at most $(1-p)^K+2^{-K+1}\leq e^{-\frac{3}{2}\alpha N\log{n}}$, where we used~\eqref{eq:inter} and that $1-p\leq e^{-p}$.

To conclude the proof of the lemma, 
we use a union bound over all partitions $(A,B)$ of $R$. 
Since for every $1\leq a \leq N/2$, 
there are $\binom{N}{a}\leq e^{a\log{N}}$ partitions with $|A|=a$. 
Conditional on $\delta(G[R])\geq200\log{n}/p$, 
the probability that $G^\cD_p[R]$ is disconnected is at most
$$
\sum_{a=1}^{N/2} \sum_{R=A\cup B \atop |A|=a} \Pro[e(G^\cD_p[A,B])=0]
\leq \sum_{a=1}^{N/2} e^{a\log{N}} e^{-\frac{3}{2}a\log{n}} 
\leq \sum_{a=1}^{N/2} e^{-\frac{1}{2}a \log{n}} =o(1)\;.
$$
\end{proof}

\subsection{Proof of Proposition~\ref{prop:very_rob}}

In this section we use the results from the two previous subsections to conclude the proof of Proposition~\ref{prop:very_rob}.
Let $p,\delta, \epsilon$ and $\avd$ be as in the statement.
Let $n$ be large enough in terms of these parameters. 
Let $V$ be a set of size $n$.
Let $\cD$ be a degree sequence on $V$ with $\Td\leq \avd n$. 
Recall that $S_1 = \{ u \in V \ : \ d(u)\geq \log^2{n} \}$,
$S_2 = \{ u \in V \ : \ d(u)\geq n^{1/3} \}$ and  $S_3  = \{ u \in V \ : \ d(u)\geq n^{4/5}\}$. 
Proposition~\ref{prop:very_rob} assumes that $\cD$ satisfies~\eqref{eq:De1}.
\medskip

\noindent
\emph{Case 1:} Suppose $\cD$ also satisfies \eqref{eq:De}, that is, $\sum_{u \in S_3}d(u)\geq \epsilon n/10$.
Let $s\geq 100$ be the smallest integer such that $\delta>2c^{s}$, 
where $c$ is the constant given by Corollary~\ref{cor:A} for our choice of $p$ and $\avd$. 
Set $\gamma_1:=\epsilon p/(20 s)$. 
If $|S_3| <  s$,  
then there exists a vertex $u\in S_3$ with $d(u)\geq 2\gamma_1 n/p$,
because $\cD$ satisfies \eqref{eq:De}.
This implies, by a simple application of Chernoff's inequality, 
that $G^\cD_p$ contains a star of order $\gamma_1 n$ with centre $u$,
in particular, $G^\cD_p$ contains a component of order at least $\gamma_1 n$.

Suppose now that $|S_3|\geq s$.
Let $\cA_1$ be the event that $G^\cD_p[S_3]$ is connected.
Then, by definition of $s$ and by Corollary~\ref{cor:A},
\begin{align}\label{eq:A1}
	\Pro[\overline{\cA_1}]\leq \frac{\delta}{2}.
\end{align}
Let $\cA_2$ be the event that every vertex in $S_2\sm S_3$ has a neighbour in $S_3$
and let $\cA_3$ be the event that $|N(S_2)\cap T|\leq \epsilon^2n$.
Then by Lemmas~\ref{lem:connectedtoS1} and~\ref{lem:N(S_2)T},
\begin{align*}
	\Pro[\overline{\cA_2}\cup \overline{\cA_3}]\leq 2n^{-1}.
\end{align*}
Let $\gamma_2:=p^2\epsilon^2n/3$.
We will show that $\Pro[L_1(G^\cD_p)> \gamma_2 n \mid \cA_2, \cA_3]\geq 1-\delta$.

If $|N(S_3)\cap T|\geq \epsilon^2n/2$,
then a straightforward application of Chernoff's inequality
combined with \eqref{eq:A1} shows that there is a component of order at least $p\epsilon^2n/3\geq \gamma_2 n$ in $G_p^\cD$
with probability at least $1-\delta$.

If $|N(S_3)\cap T|\leq \epsilon^2n/2$,
then $|N(S_2\sm S_3)\cap T|\geq \epsilon^2n/2$.
Let $F$ be a 
forest in $G^{\cD}$ such that $F$ contains $N(S_2\sm S_3)\cap T$, for every vertex $x_1\in N(S_2\sm S_3)\cap T$,
there is a path $x_1x_2x_3$ in $F$
such that $x_2\in S_2\sm S_3$ and $x_3\in S_3$, and among all such forests, $F$ contains as few as edges as possible.
To complete the case when $\cD$ satisfies \eqref{eq:De},
we will show that $G^\cD_p[S_3]\cup F_p$ contains a component of order at least $\gamma_2n$ with probability at least $1-\delta$. 
Consider a realisation of $G^{\cD}$ that satisfies $\cA_2 \cap \cA_3$.
Observe first that whether a certain edge in $F$ is present in $F_p$
changes the number of vertices in $N(S_2\sm S_3)\cap T$ that are connected via $F_p$ to $S_3$ by at most $n^{4/5}$.
Thus assuming $\cA_1$ holds,
a straightforward application of McDiarmid's inequality (Lemma~\ref{lem:McDineq})
shows that 
there is a component of order at least $p^2\cdot \epsilon^2n/3=\gamma_2 n$ in $G^\cD_p[S_3]\cup F_p$
with probability at least $1-n^{-1}$.
This together with \eqref{eq:A1}, completes the case when $\cD$ satisfies \eqref{eq:De}.

\medskip

\noindent
\emph{Case 2:} Now, suppose that $\cD$ does not satisfy Condition $(D^3_{\epsilon})$. 
Since it satisfies~\eqref{eq:De1}, by Lemma~\ref{lem:large min deg}, we obtain
\begin{align*}
	\Pro\left[\delta(G^\cD[S_1])\geq \frac{\epsilon}{16\avd}\log^2{n}\right]=1-o(1).
\end{align*}
Together with Lemma~\ref{lem:connec2} where $S_1$ plays the role of $R$, 
we conclude that 
\begin{align}\label{eq:GScon}
	\Pro[G^\cD_p[S_1]\ \text{\rm is connected}]=1-o(1).
\end{align}
In order to show that $G^\cD_p$ contains a giant component, 
we will show that $|N(S_1)\cap T|$ is large.
Let $K:=\lfloor \epsilon n / (128\avd)\rfloor$. 
For every $0\leq k\leq 2K$,  
let $\cF_k$ be the set of graphs with degree sequence $\cD$ such that $|N(S_1)\cap T|=k$.

Let $G\in \cF_k$. 
Using~\eqref{eq:sizeT} and $\delta(G)\geq 1$, 
there are at least $|T|-k\geq 2n/3-k \geq n/2$ choices for an edge $xy$ with $x\in T\sm N(S_1)$.
Observe that $d(y)\leq \log^2 n$, since $x\notin N(S_1)$. 
Also, since $x,y\notin S_1$ and $\cD$ does not satisfy \eqref{eq:De}, 
we claim that there are at most 
$$3 \avd \log^2 n +n^{4/5}\log^2{n}+\epsilon n/10\leq \frac{\epsilon n}{5}$$ 
edges incident to a neighbour of either $x$ or $y$. 
Indeed, the number of edges incident to a neighbour of $x$ is bounded by $3 \avd \log^2 n$, as $x$ has no 
neighbours inside $S_1$. 
Now, the neighbours of $y$ are classified either as the neighbours that belong to $S_3$ or those that do not.  
Since property \eqref{eq:De} does not hold, and 
there are at most $\epsilon n/10$ edges incident to any vertex in $S_3$, there are at most $\epsilon n /10$ edges 
incident to the first class of neighbours. 
Regarding the latter class of neighbours, there are at most $\log^2 n$ of them (as $y\not \in S_1$) 
and each has degree at most 
$n^{4/5}$. Thereby, there at most $n^{4/5}\cdot \log^2 n$ such edges. 
Hence our claim holds.

Let $uv$ be an edge such that $u\in S_1$, $v\notin N(y)$, and
either $v\notin T$ or if $v\in T$, then there exists a $u'\in S_1$ with $u'v\in E(G)$.  
Since $\sum_{u \in S_1}d(u)\geq \epsilon n$, 
there are at least $\epsilon n- k-{\epsilon}n/5\geq {\epsilon}n/{2}$ such edges. 
Hence, the total number of $\{xy,uv\}$-switches that transform $G$ into a graph 
in $\cF_{k+1}\cup \cF_{k+2}$ is at least $\epsilon n^2/4$
(we transform $G$ into a graph satisfying $\cFkkk$ if $v\in S_1$ and $y\in T\sm N(S_1)$).

If $G \in \cFkk\cup \cFkkk$, 
then there are at most $(k+2)\avd n$ switches that transform $G$ into a graph in $\cFk$. 
As before, for every $0\leq k\leq  2K-2$, this implies
\begin{align*}
	\Pro[\cF_k]\leq \frac{4(k+2)\avd n}{\epsilon n^2}\cdot (\Pro[\cFkk]+\Pro[\cFkkk]) \leq \frac{1}{4}\cdot \max\{\Pro[\cFkk],\Pro[\cFkkk]\} \;.
\end{align*}
Therefore, 
$$
\sum_{k=0}^{K-1}\Pro[ \cFk] \leq 2^{-K} \Pro[\cF_{2K}]\leq 2^{-K}=o(1)\;.
$$
Hence 
\begin{align*}
	\Pro[|N(S_1)\cap T|\geq \lfloor \epsilon n / (128\avd)\rfloor]=1-o(1)\;.
\end{align*}
Let $\gamma_3:=\epsilon p/(130\avd)$.
The Chernoff bound (Lemma~\ref{lem:Chernoff}) implies that
\begin{align*}
	\Pro[|N_{G^\cD_p}(S_1)\cap T|\geq p \epsilon n / (130\avd)]=1-o(1)\;.
\end{align*}
Together with \eqref{eq:GScon}, this implies $\Pro[L_1(G^\cD_p)\geq \gamma_3 n]\geq 1-\delta$.
Setting $\gamma:=\min\{\gamma_1,\gamma_2,\gamma_3\}$, we obtain 
$$
\Pro[L_1(G^\cD_p)\geq \gamma n]\geq 1-\delta\;.
$$

 \section{Sequences of degree sequences: proof of Theorem~\ref{thm:seq}} \label{sec:sequ}

Let $\fD=(\cD_n)_{n\geq 1}$ be a sequence of degree sequences with $\cD_n=(d_1^{(n)},\dots, d_n^{(n)})$. 
For the sake of simplicity, we write $\cD_n=(d_1,\dots, d_n)$ and $W(c):=W(c,\cD_n)$.
Set $$d_{c,n} := \max\left\{\frac{\sum_{i \in V\sm W(c)}d_i (d_i-1)}{\sum_{i \in V \sm W(c)} d_i},1\right\}.$$
We assume that $d_c:=\lim_{n \to \infty} d_{c,n}$ exists for every $c\geq 1$
and that $d$ is such that
$$
d= \sup_{c \geq 1} d_c=\sup_{c \geq 1} \lim_{n \to \infty} d_{c,n} \in [1,\infty) \;. 
$$
We define the critical probability as in~\eqref{eq:crit_prob} by
\begin{align*}
p_\cri (c,\cD_n) = \min \left\{ \frac{\sum_{i \in V\setminus W(c)} d_i}{\sum_{i\in V\setminus W(c)} d_i (d_i-1)},1\right\} = \frac{1}{d_{c,n}}\;.
\end{align*} 
We start with the proof of part (i)
and begin with a claim which states that for every large $c$
we can replace $1/d$ by $p_\cri(c,\cD_n)$ provided $n$ is large enough in terms of $c$.

\begin{claim}\label{cla:bounds}
For every $\epsilon \in (0,1/2)$, 
there exists $c_\epsilon$ such that for every $c\geq c_\epsilon$, there exists $n_{\epsilon,c}$ such that for every $n\geq n_{\epsilon,c}$, we have
\begin{itemize}
\item[-] if $p<(1-\epsilon)\frac{1}{d}$, then  $p<\left(1-\frac{\epsilon}{4}\right)p_\cri(c,\cD_n)$.
\item[-] if $p>(1+\epsilon)\frac{1}{d}$, then  $p>\left(1+\frac{\epsilon}{4}\right)p_\cri(c,\cD_n)$.
\end{itemize}
\end{claim}
\begin{proof}
Note first that $d_{c,n}$ is non-decreasing with respect to $c$;
that is, $d_{c_2,n}\geq d_{c_1,n}$ for $c_2\geq c_1$.
This implies that $d_{c_2}=\lim_{n \to \infty} d_{c_2,n}\geq \lim_{n \to \infty} d_{c_1,n} =d_{c_1}$.
Hence $(d_c)_{c\geq 1}$ is a monotone non-decreasing
sequence and it converges to $d$.
Furthermore, $d<\infty$ by assumption.
Thus, for any $\epsilon >0$, there exists $c_{\epsilon}$ such that for any $c > c_{\epsilon}$, we have 
$$
(1 - \epsilon^2/2)d < d_c \leq d.
$$
In turn, given $c$, there exists $n_{\epsilon, c}$ such that for any $n > n_{\epsilon,c}$, we have 
\begin{equation*}
(1-\epsilon^2/2)d_c  < d_{c,n} < (1+\epsilon^2)d_c. 
\end{equation*}
Therefore, for every $c\geq c_\epsilon$ and every $n\geq n_{\epsilon, c}$ we directly obtain
\begin{equation}\label{eq:limit}  
(1-\epsilon^2)d  < d_{c,n} < (1+\epsilon^2)d. 
\end{equation}
Moreover, if $\epsilon<1/2$ and $p< (1-\epsilon) \frac{1}{d}$, then 
\begin{align*}
p <(1 - \epsilon)(1 + \epsilon^2)\frac{1}{d_{c,n}} < (1- \epsilon/4) \frac{1}{d_{c,n}}.
\end{align*} 
Similarly, if $\epsilon < 1/2$ and 
$p> (1 + \epsilon) \frac{1}{d}$, then 
\begin{align*}
p > (1+ \epsilon)(1-\epsilon^2) \frac{1}{d_{c,n}} > (1+\epsilon /4) \frac{1}{d_{c,n}}.
\end{align*}
\end{proof}
\noindent In what follows we will select $c_1$, $c_2$ and $\eta$ such that the hypotheses of Theorem~\ref{thm:thres} are satisfied.
By the strong uniform integrability assumption, for any $\epsilon>0$, 
there exists a $c'_{\epsilon}\in \N$ such that 
for every $c\geq c'_{\epsilon}$, there exists $n'_{\epsilon,c}$ such that for $n\geq n'_{\epsilon,c}$ we have
\begin{equation}\label{eq:limit2}  
\sum_{j\in W(c)} d_j\leq \frac{\epsilon^2}{c} \cdot n\;.
\end{equation}

We may assume that for fixed $\epsilon$ and all $c_1\leq c_2$,
we have $n_{\epsilon,c_1}\leq n_{\epsilon,c_2}$ and $n_{\epsilon,c_2}'\leq n_{\epsilon,c_2}$
as we simply can replace $n_{\epsilon,c_2}$ by $\max_{c'\leq c_2}\{n_{\epsilon,c'},n_{\epsilon,c'}'\}$.
We may also assume that $\epsilon < (64 d\avd)^{-1}$.
We choose $c_1:= \max\{c_\epsilon,c_{\epsilon}'\}$.
Suppose $c>c_1$ and $n\geq n_{\epsilon,c}$. 
Next we prove that $A_2(\epsilon/4, c_1,c)$ holds for a suitable $c$.
Note that this condition is only needed in Theorem~\ref{thm:thres}~(ii).
If $d_{c_1,n}\leq (1+\epsilon/5)$, then $p>(1+\epsilon)/d$ implies (by Claim~\ref{cla:bounds} and $c_1\geq c_\epsilon$) that
$$
p>\left(1+\frac{\epsilon}{4}\right)\frac{1}{d_{c_1,n}}>1,
$$
and there is nothing to prove. 
Thus, we may assume that $d_{c_1,n} > (1+\epsilon/5)$. 
Hence, $d_{c,n}= \frac{\sum_{i \in V\sm W(c)}d_i (d_i-1)}{\sum_{i \in V \sm W(c)} d_i}\leq c$ for any $c\geq c_1$ and $n$ sufficiently large in terms of $c$.
It follows that
\begin{eqnarray*}
\sum_{j \in W(c_1) \sm W(c)} d_j(d_j-1) 
&=& \sum_{j \in V\sm W(c)} d_j (d_j-1) - \sum_{j \in V\sm W(c_1)} d_j (d_j-1)  \\
& =& d_{c, n} \sum_{j \in V\sm W(c)}d_j    - d_{c_1,n} \sum_{j \in V\sm W(c_1)}d_j \\
&=& (d_{c,n} - d_{c_1,n}) \sum_{j \in V\sm W(c_1)}d_j+ d_{c,n}\sum_{j \in W(c_1)\sm W(c)}d_j \\
&\stackrel{\eqref{eq:limit2}}{\leq} &(d_{c,n} - d_{c_1,n}) \avd n +\epsilon^2dn\\
&\stackrel{\eqref{eq:limit}}{\leq}& 3\epsilon^2 d \avd n\;.
\end{eqnarray*}
This in turn implies that 
\begin{eqnarray*} 
\sum_{j \in W(c_1) \sm W(c)} d_j^2 
&=&\sum_{j \in W(c_1) \sm W(c)} d_j (d_j-1) +\sum_{j \in W(c_1) \sm W(c)} d_j \\
&\leq &3\epsilon^2 d \avd n + \sum_{j \in W(c_1)} d_j \\
&\stackrel{\eqref{eq:limit2}}{<} &4\epsilon^2  d \avd n  \\
&\leq &\frac{\epsilon/4}{4}\cdot n \;.
\end{eqnarray*}
Thus $A_2 (\epsilon/4,c_1,c)$ holds for all $c\geq c_1$ and $n \geq n_{\epsilon,c}$.
Note that $c_1$ only depends on $\epsilon$.

Let $\eta = \eta (\gamma, \epsilon/4, \avd)$ be as in Theorem~\ref{thm:thres}. 
Using again Condition~$(b)$, for $n\geq n_{\eta,c_2}$ we have
\begin{equation*}
\sum_{j\in W(c_2)} d_j \leq \frac{\eta}{c_2}\cdot n\;,
\end{equation*}
and thus $A_1(\eta,c_2)$ is satisfied (even if $d_{c_1,n}\leq (1+\epsilon/5)$).

Also let  $\rho=\rho(\epsilon/4,c_1)$ be the constant provided by Theorem~\ref{thm:thres}, which in this case only depends on $\epsilon$;
that is, we can choose $\gamma\leq \rho$.
Let $n$ be larger than $\max\{n_{\eta,c_2},n_{\epsilon,c_2}\}$
and the $n_0$ given by Theorem~\ref{thm:thres} for the parameters $\epsilon/4,\gamma, c_1,c_2,\avd$.
By Claim~\ref{cla:bounds}, we can apply Theorem~\ref{thm:thres} with $\epsilon/4$ to $\cD_n$ to conclude that
\begin{align*} 
&\mbox{if $p < (1- \epsilon) \frac{1}{d}$, then $\Pro [L_1 (G^{\cD_n}_p) > \gamma n] = \new{o}(1)$}, \\
&\mbox{if $p > (1+ \epsilon) \frac{1}{d}$, then $\Pro [L_1 (G^{\cD_n}_p) > \rho n] =1-  \new{o}(1)$}.
\end{align*} 
We proceed to the proof of Theorem~\ref{thm:seq}~(ii). 
Our aim is to apply Theorem~\ref{thm:rob}. 
Let us check that the hypotheses are satisfied.  
Given $\delta,p$ and $\avd$, let $K$ be the constant provided by Theorem~\ref{thm:rob}.

Suppose first that $d=\infty$. 
Since $\sup_{c \geq 1} d_c=\lim_{c\to \infty} d_c= d=\infty$, 
there exists $c_K$ such that for every $c\geq c_K$, we have 
$$
\lim_{n\to \infty} d_{c,n}= d_c >2K\;.
$$

Similarly, there exists $n_{K,c}$ such that $d_{c,n}\geq K$ for every $n\geq n_{K,c}$. 
For $c\geq \max\{c_{K},2\avd\}$ and $n\geq n_{K,c}$, we obtain
$$
 \sum_{j \in V\sm W(c)} d_j^2  \geq \sum_{j \in V\sm W(c)} d_j (d_j - 1) \geq K \sum_{j\in V\sm W(c)} d_j \geq \frac{K}{2}\cdot n\;,
$$
and so $A_2(K,0,c)$ does not hold.
Thus Theorem~\ref{thm:rob} leads to the desired conclusion.

Suppose now that $d< \infty$ but the sequence robustly fails the strong uniform integrability assumption. 
Let $c_0$ be such that $f(c)\geq K$ for every $c\geq c_0$. 
As $f(c)\to \infty$ as $c\to \infty$ such a $c_0$ exists. 
This in turn immediately implies that $A_1(K,c)$ does not hold provided $n$ is large enough.
Again, Theorem~\ref{thm:rob} leads to the desired conclusion and this completes the proof.
\medskip

We close this section with the following remark.
Suppose that $\lim_{n\to \infty} n_i/n =: \lambda_i < \infty$ for all $i\geq 1$,  that
$\sum_{i \geq 1}\lambda_i=1$, and that
$\sum_{i\geq 1}i\lambda_i <\infty$. 
Then 
$$ d = \frac{\sum_{i\geq 1} i(i-1)\lambda_i}{\sum_{i\geq 1} i\lambda_i}. $$
This recovers the results obtained by the first author~\cite{fountoulakis2007percolation}, Janson~\cite{janson2008percolation}, and Bollob\'as and Riordan~\cite{bollobas2015old}.

\section{Application: power-law degree distributions}\label {sec:PL}

Power law degree distributions have attracted considerable interest as they are one of the usual
characteristics of complex networks~\cite{albert2002statistical}.
Roughly speaking, in such degree sequences the 
fraction of vertices that have degree equal to $k$ (when $k$ is large) scales like $k^{-\gamma}$, for some 
$\gamma >0$. 

A variety of random graph models which exhibit a power law degree distribution have been introduced in the last 
15 years, mainly, in search for a sound model for complex networks. 
Among other properties, \emph{robustness} is a central property that has been considered in this context; 
that is, how robust a random network is if several of its edges or its vertices fail. 

In several random graph models with a power law degree distribution, it has been observed that if 
$\gamma>3$, 
then there exists a critical value $p_\cri$ (which is bounded away from $0$) for the appearance of a giant component in the bond percolation process.
However, 
if $\gamma \leq 3$, for any fixed $p>0$ (that is, independent of the order of the random graph), a giant 
component survives the random deletions with high probability. 
This behaviour has been observed in diverse random graph models that give rise to power-law degree
distributions such as the configuration model~(\cite{firstpassage2015+}, Corollary 2.5), the preferential attachment model~\cite{bollobasrior} and random graphs on the hyperbolic plane~\cite{hyper_perc2015}.

We now apply Theorem~\ref{thm:seq} in this context. This recovers a known result for power law sequences but also exemplifies how our results can be used for particular degree sequences.
Consider a sequence of degree sequences $(\cD_n)_{n \in \mathbb{N}}$,
where $\cD_n$ is a feasible degree sequence on $[n]$ and assume that it satisfies the following: for $k\geq 1$, let $n_k$ denote the number of vertices of degree $k$ in $\cD_n$, 
then there exist positive constants $\gamma,\lambda_1, \lambda_2, k_0>0$ such that for every $k \geq k_0$, we have 
$$
\frac{\lambda_1}{k^{\gamma}} \leq \frac{n_k}{n} \leq \frac{\lambda_2}{k^{\gamma}}\;.
$$
If so, we say that $\cD_n$ follows a power law distribution with exponent $\gamma$.

In this section we show that power law distributions, as defined here, show the same behaviour around $\gamma=3$. 
As before, we write $\cD_n=(d_1,\dots, d_n)$ and $W(c):=W(c,\cD_n)=\{ i \ : \  d_i \geq c \}$ for every $c\geq 1$.

Let $\cD_n$ follow a power law distribution with $\gamma>3$. Then, there exists $\lambda_2' >0$ such that for every $c_2 \geq k_0$, we have
$$\sum_{i\in W(c_2)} d_i 
= \sum_{k\geq c_2} k n_k
\leq \lambda_2 n \sum_{k\geq c_2} k^{1-\gamma} 
\leq \frac{\lambda_2'}{c_2^{\gamma -2}}\cdot n = \lambda_2' c_2^{3-\gamma}\cdot \frac{n}{c_2} \;,$$
and thus, $\cD_n$ satisfies $A_1(\lambda_2' c_2^{3-\gamma},c_2)$. 

Moreover, there exists $\lambda_2''>0$ such that for all $c_2\geq c_1 \geq k_0$, we have 
$$
\sum_{i\in W(c_1)\sm W(c_2)} d_i^2 =\sum_{k =c_1}^{c_2-1} k^2 n_k \leq \lambda_2 n \sum_{k =c_1}^{c_2-1} k^{2-\gamma} \leq  \lambda_2''c_1^{3-\gamma }\cdot n\;,
$$
that is, $\cD_n$ satisfies $A_2 (4\lambda_2''c_1^{3-\gamma},c_1,c_2)$. 

Provided that $c_1$ and $c_2$ are large enough and $\gamma>3$~(so the first parameters in conditions $A_1$ and $A_2$ are arbitrarily small), we can apply Theorem~\ref{thm:thres} to determine a quantity $p_\cri>0$ that is bounded away from 0, such that bond percolation in $G^{\cD_n}$ has a threshold at $p_\cri$.

Now, let $\cD_n$ follow a power law distribution with $2<\gamma<3$. Then, there exists $\lambda_1'>0$ such that for every $c\geq k_0$, we have
$$
 \sum_{i\in W(c)} d_i = \sum_{k\geq c} k n_k 
\geq \lambda_1 n \sum_{k\geq c} k^{1-\gamma} \geq \lambda_1' c^{2-\gamma} \cdot n= \lambda_1' c^{3-\gamma} \cdot \frac{n}{c}\;,$$
that is, $\cD_n$ does not satisfy $A_1(\lambda_1'c^{3-\gamma},c)$.

Provided that $c_1$ is large enough (so the first parameter in condition $A_1$ is arbitrarily large), we can apply Theorem~\ref{thm:rob} to show that bond percolation in $G^{\cD_n}$ does not have a positive threshold.

Note that if $\gamma\leq 2$, then the average degree of $G^{\cD_n}$ is unbounded and our results do not apply.

We finally state the ``limit'' version of the result for $\cD_n$ that follows a power law distribution. Suppose that there exists $c>0$ such that for all $k \geq 1$, we have 
$$
\lim_{n \to \infty} \frac{n_k}{n} = c k^{-\gamma}\;.
$$
If $\gamma>3$, then $d < \infty$, while if $\gamma<3$, then $d=\infty$.
So, Theorem~\ref{thm:seq} implies that in the former case we have $p_\cri = 1/d>0$, whereas in the latter case $p_\cri =0$. 

It is worth to stress that our results do not provide any meaningful information at $\gamma=3$.

\section{Concluding remarks}\label{sec:remarks}
We finish the paper with some remarks on our results.
\begin{itemize}
%

\item[1)] Theorem~\ref{thm:rob} provides a statement that holds only with probability at least $1-\delta$. The only part of its proof that does not hold with high probability is Corollary~\ref{cor:A}. 
This makes it easy to construct degree sequences that show that this cannot be improved.
For a given $\rho>0$, let us consider the following degree sequence on $n$ vertices (large enough in terms of $\rho$).
Let $a:=\lfloor 2/\rho \rfloor$ and suppose $a$ divides $n-a$.
Consider the degree sequence with $a$ vertices of degree 
$(n-a)/a + a -1$, and $n-a$ vertices of degree~$1$. 
This degree sequence is feasible and the only graph (up to isomorphism) with this degree sequence consists of a clique of size $a$ where each of its vertices is adjacent to  $n/a-1$ vertices of degree~$1$. 
With positive probability independently of $n$ all $\binom{a}{2}$ edges inside the clique of size $a$ fail to percolate in $G_p^\cD$. 
If so, $L_1(G_p^\cD)\leq \rho n/2$.
Thus for every $p\in [0,1)$, we have
$$
\Pro[L_1(G_p^\cD)<\rho n] >\delta(\rho,p)\;.
$$
Observe that these degree sequences also do not satisfy $A_1(K,c)$ for all $c\geq 2K$.

\item[2)] In~\cite{joos2016how}, a special role is given to vertices of degree $2$. 
However, by considering bond percolation this special situation never appears.
If most of the edges are incident to vertices of degree $2$ after the bond percolation, then $p\approx 1$ and almost all vertices have degree~$2$ already before the percolation.
In this case set $p_\cri:=1$. 
Let $W$ be the set of vertices with degree different from $2$. 
If $\sum_{i\in W} d_i = o(n)$, then 
$|N[W]|=o(n)$. 
For every $\epsilon>0$ and every $p< 1-\epsilon$, it follows that, 
$$
\sum_{i\in V\setminus N[W]} d_i(p(d_i-1)-1) = (n-|N[W]|)2(p-1)< - \epsilon n\;.
$$
Using the first part of Proposition~\ref{prop:explo} we obtain that $G^\cD_p$ has no giant component with high probability, and thus $p_\cri=1$. 
%
%

\item[3)] The previous remark is a particular case of the case $\Td/n\to \infty$. 
While it might seem natural that $p_\cri(\cD)\to 0$, here we provide an example for which $\Td/n\to \infty$ and $p_\cri$ is bounded away from $0$. 

Consider the degree sequence $\cD$ formed by $n^{2/3}$ vertices of degree $n^{2/3}$ and $n-n^{2/3}$ vertices of degree $1$. The critical condition in~\cite{joos2016how} shows that $G^\cD$ has a giant component with high probability. However, it is easy to see that, with high probability, $G^\cD_p$ has at least $(1-2p)n$ isolated vertices and thus we cannot expect to have a component of order larger than $2pn$. If $p\to 0$ (as $n\to \infty$), then $G^\cD_p$ does not have a giant component with high probability.
\end{itemize}

\bibliographystyle{amsplain}
\bibliography{percol_ref}

\vfill

\small
\vskip2mm plus 1fill
\noindent
Version \today{}
\bigbreak

\noindent
Nikolaos Fountoulakis\\
{\tt <n.fountoulakis@bham.ac.uk>}\\
School of Mathematics, University of Birmingham, Birmingham\\
United Kingdom\\

\noindent
Felix Joos\\
{\tt <joos@informatik.uni-heidelberg.de>}\\
Institute for Computer Science, Heidelberg University, Heidelberg\\
Germany\\

\noindent
Guillem Perarnau\\
{\tt <guillem.perarnau@upc.edu>}\\
IMTech, Universitat Polit\`ecnica de Catalunya, and Centre de Recerca Matem\`atica, Barcelona\\ 
 Spain

\end{document}